\documentclass[11pt,letterpaper]{article}

\usepackage{graphicx,subfigure}

          \newcommand{\rig}{\mathsf{rig}} \newcommand{\chop}{\mathsf{chop}} 

\newcommand{\shatter}{\mathsf{shatter}} \newcommand{\shards}{\mathsf{shards}} \newcommand{\len}{\mathrm{len}}
 \newcommand{\vcong}{\mathfrak{v}}  \newcommand{\cross}{\chi}   \newcommand{\matousek}{Matou{\v{s}}ek } \newcommand{\spread}{\mathfrak{s}} \newcommand{\sobs}{\mathfrak{s}_{\mathrm{obs}}} \newcommand{\csobs}{\bar{\mathfrak{s}}_{\mathrm{obs}}} \newcommand{\cspread}{\bar{\mathfrak{s}}}

\def\stocmode{0} \def\jamesmode{0} \def\arxivmode{0} \def\fastmode{0}   \def\showauthornotes{0}  \def\showkeys{0} \def\showdraftbox{1} \def\showcolorlinks{1} \def\usemicrotype{1} \def\showfixme{1}

\ifnum\arxivmode=0 \newcommand{\vvS}{\vvmathbb{S}} \newcommand{\vvB}{\vvmathbb{B}} \else \newcommand{\vvS}{\mathbb{S}} \newcommand{\vvB}{\mathbb{B}} \fi

\ifnum\fastmode=0
\usepackage{etex}
\fi

\ifnum\fastmode=0
\usepackage[l2tabu, orthodox]{nag}
\fi

\usepackage{xspace,enumerate}

\usepackage[dvipsnames]{xcolor}

\ifnum\fastmode=0
\usepackage[T1]{fontenc}
\usepackage[full]{textcomp}
\fi

\usepackage[american]{babel}

\usepackage{mathtools}

\usepackage{amsthm}

\newtheorem{theorem}{Theorem}[section]
\newtheorem*{theorem*}{Theorem}

\newtheorem*{proposition*}{Proposition}
\newtheorem{lemma}[theorem]{Lemma}
\newtheorem*{lemma*}{Lemma}
\newtheorem{corollary}[theorem]{Corollary}
\newtheorem*{conjecture*}{Conjecture}
\newtheorem{fact}[theorem]{Fact}
\newtheorem*{fact*}{Fact}

\newtheorem*{exercise*}{Exercise}

\newtheorem*{hypothesis*}{Hypothesis}

\theoremstyle{definition}
\newtheorem{definition}[theorem]{Definition}

\newtheorem{example}[theorem]{Example}
\newtheorem{question}[theorem]{Question}

\newtheorem{exercise-easy}[theorem]{Exercise}
\newtheorem{exercise-med}[theorem]{Exercise}
\newtheorem{exercise-hard}[theorem]{Exercise$^\star$}
\newtheorem{claim}[theorem]{Claim}
\newtheorem*{claim*}{Claim}

\newtheorem{remark}[theorem]{Remark}
\newtheorem*{remark*}{Remark}

\newtheorem*{observation*}{Observation}

\usepackage[
letterpaper,
top=0.7in,
bottom=0.9in,
left=1in,
right=1in]{geometry}

\ifnum\arxivmode=0
\usepackage{newpxtext} 
\usepackage{textcomp} 
\usepackage[varg,bigdelims]{newpxmath}
\usepackage[scr=rsfso]{mathalfa}
\usepackage{bm} 
\let\mathbb\varmathbb
\fi

\ifnum\arxivmode=1
\usepackage[varg]{pxfonts} 
\usepackage{textcomp} 
\usepackage[scr=rsfso]{mathalfa}
\usepackage{bm} 
\fi

\ifnum\showkeys=1
\usepackage[color]{showkeys}
\fi

\ifnum\arxivmode=1
\makeatletter
\@ifpackageloaded{hyperref}
{}
{\usepackage{hyperref}
\PassOptionsToPackage{pagebackref=true}{hyperref}
}
\else
\makeatletter
\@ifpackageloaded{hyperref}
{}
{\usepackage[pagebackref]{hyperref}}
\fi

\ifnum\showcolorlinks=1
\hypersetup{
pagebackref=true,
colorlinks=true,
urlcolor=blue,
linkcolor=blue,
citecolor=OliveGreen,
linktocpage=false}
\fi

\ifnum\showcolorlinks=0
\hypersetup{
colorlinks=false,
pdfborder={0 0 0}}
\fi

\usepackage{prettyref}

\newcommand{\savehyperref}[2]{\texorpdfstring{\hyperref[#1]{#2}}{#2}}

\newrefformat{eq}{\savehyperref{#1}{\textup{(\ref*{#1})}}}
\newrefformat{lem}{\savehyperref{#1}{Lemma~\ref*{#1}}}
\newrefformat{def}{\savehyperref{#1}{Definition~\ref*{#1}}}
\newrefformat{thm}{\savehyperref{#1}{Theorem~\ref*{#1}}}
\newrefformat{cor}{\savehyperref{#1}{Corollary~\ref*{#1}}}
\newrefformat{cha}{\savehyperref{#1}{Chapter~\ref*{#1}}}
\newrefformat{sec}{\savehyperref{#1}{Section~\ref*{#1}}}
\newrefformat{app}{\savehyperref{#1}{Appendix~\ref*{#1}}}
\newrefformat{tab}{\savehyperref{#1}{Table~\ref*{#1}}}
\newrefformat{fig}{\savehyperref{#1}{Figure~\ref*{#1}}}
\newrefformat{hyp}{\savehyperref{#1}{Hypothesis~\ref*{#1}}}
\newrefformat{alg}{\savehyperref{#1}{Algorithm~\ref*{#1}}}
\newrefformat{rem}{\savehyperref{#1}{Remark~\ref*{#1}}}
\newrefformat{item}{\savehyperref{#1}{Item~\ref*{#1}}}
\newrefformat{step}{\savehyperref{#1}{step~\ref*{#1}}}
\newrefformat{conj}{\savehyperref{#1}{Conjecture~\ref*{#1}}}
\newrefformat{fact}{\savehyperref{#1}{Fact~\ref*{#1}}}
\newrefformat{prop}{\savehyperref{#1}{Proposition~\ref*{#1}}}
\newrefformat{prob}{\savehyperref{#1}{Problem~\ref*{#1}}}
\newrefformat{claim}{\savehyperref{#1}{Claim~\ref*{#1}}}
\newrefformat{relax}{\savehyperref{#1}{Relaxation~\ref*{#1}}}
\newrefformat{red}{\savehyperref{#1}{Reduction~\ref*{#1}}}
\newrefformat{part}{\savehyperref{#1}{Part~\ref*{#1}}}
\newrefformat{ex}{\savehyperref{#1}{Exercise~\ref*{#1}}}
\newrefformat{ques}{\savehyperref{#1}{Question~\ref*{#1}}}

\newcommand{\Sref}[1]{\hyperref[#1]{\S\ref*{#1}}}

\usepackage{nicefrac}

\ifnum\fastmode=0
\ifnum\usemicrotype=1
\usepackage{microtype}
\fi
\fi

\ifnum\showauthornotes=1
\newcommand{\Authornote}[2]{{\sffamily\small\color{blue}{[#1: #2]}}\medskip}
\newcommand{\Authornotecolored}[3]{{\sffamily\small\color{#1}{[#2: #3]}}}
\newcommand{\Authorcomment}[2]{{\sffamily\small\color{gray}{[#1: #2]}}}
\newcommand{\Authorstartcomment}[1]{\sffamily\small\color{gray}[#1: }

\newcommand{\Authorfnote}[2]{\footnote{\color{red}{#1: #2}}}
\newcommand{\Authorfixme}[1]{\Authornote{#1}{\textbf{??}}}
\newcommand{\Authormarginmark}[1]{\marginpar{\textcolor{red}{\fbox{\Large #1:!}}}}
\else
\newcommand{\Authornote}[2]{}
\newcommand{\Authornotecolored}[3]{}
\newcommand{\Authorcomment}[2]{}
\newcommand{\Authorstartcomment}[1]{}

\newcommand{\Authorfnote}[2]{}
\newcommand{\Authorfixme}[1]{}
\newcommand{\Authormarginmark}[1]{}
\fi

\ifnum\showfixme=0

\fi

\usepackage{boxedminipage}

\newcommand{\Esymb}{\mathbb{E}}
\newcommand{\Psymb}{\mathbb{P}}

\DeclareMathOperator*{\E}{\Esymb}

\DeclareMathOperator*{\ProbOp}{\Psymb}

\renewcommand{\Pr}{\ProbOp}

\newcommand{\textparen}[1]{\text{(#1)}}

\ifx\because\undefined
\newcommand{\because}[1]{\textparen{because #1}}
\else
\renewcommand{\because}[1]{\textparen{because #1}}
\fi

\newcommand{\seteq}{\mathrel{\mathop:}=}

\newcommand\bdot\bullet

\ifx\mathds\undefined 
\newcommand{\Ind}{\mathbb I}
\else
\newcommand{\Ind}{\mathds 1}
\fi

\DeclareMathOperator{\vol}{vol}

\DeclareMathOperator{\argmax}{argmax}

\DeclareMathOperator{\supp}{supp}
\DeclareMathOperator{\dist}{dist}

\newcommand{\Z}{\mathbb Z}
\newcommand{\N}{\mathbb N}
\newcommand{\R}{\mathbb R}

\newcommand{\cB}{\mathcal B}

\newcommand{\cF}{\mathcal F}

\newcommand{\cH}{\mathcal H}

\newcommand{\cL}{\mathcal L}

\newcommand{\cP}{\mathcal P}

\newcommand{\cT}{\mathcal T}

\renewcommand{\leq}{\leqslant}

\renewcommand{\geq}{\geqslant}

\ifnum\showdraftbox=1

\else

\fi

\let\epsilon=\varepsilon

\numberwithin{equation}{section}

\newcommand\MYcurrentlabel{xxx}

\newcommand{\MYstore}[2]{%
  \global\expandafter \def \csname MYMEMORY #1 \endcsname{#2}%
}

\newcommand{\MYload}[1]{%
  \csname MYMEMORY #1 \endcsname%
}

\newcommand{\MYnewlabel}[1]{%
  \renewcommand\MYcurrentlabel{#1}%
  \MYoldlabel{#1}%
}

\newcommand{\MYdummylabel}[1]{}

\newcommand{\torestate}[1]{%
  \let\MYoldlabel\label%
  \let\label\MYnewlabel%
  #1%
  \MYstore{\MYcurrentlabel}{#1}%
  \let\label\MYoldlabel%
}

\newcommand{\restatetheorem}[1]{%
  \let\MYoldlabel\label
  \let\label\MYdummylabel
  \begin{theorem*}[Restatement of \prettyref{#1}]
    \MYload{#1}
  \end{theorem*}
  \let\label\MYoldlabel
}

\newcommand{\restatelemma}[1]{%
  \let\MYoldlabel\label
  \let\label\MYdummylabel
  \begin{lemma*}[Restatement of \prettyref{#1}]
    \MYload{#1}
  \end{lemma*}
  \let\label\MYoldlabel
}

\newcommand{\restateprop}[1]{%
  \let\MYoldlabel\label
  \let\label\MYdummylabel
  \begin{proposition*}[Restatement of \prettyref{#1}]
    \MYload{#1}
  \end{proposition*}
  \let\label\MYoldlabel
}

\newcommand{\restatefact}[1]{%
  \let\MYoldlabel\label
  \let\label\MYdummylabel
  \begin{fact*}[Restatement of \prettyref{#1}]
    \MYload{#1}
  \end{fact*}
  \let\label\MYoldlabel
}

\newcommand{\restate}[1]{%
  \let\MYoldlabel\label
  \let\label\MYdummylabel
  \MYload{#1}
  \let\label\MYoldlabel
}

\newcommand{\addreferencesection}{
  \phantomsection
\ifnum\stocmode=0
  \addcontentsline{toc}{section}{References}
\else
  \addcontentsline{toc}{section}{References \hspace*{1in} --------- End of extended abstract ---------}
\fi

}

\newcommand{\e}{\epsilon}

\let\origparagraph\paragraph
\renewcommand{\paragraph}[1]{\vspace*{-15pt}\origparagraph{#1.}}

\allowdisplaybreaks

\usepackage{paralist}

\usepackage{comment}

\let\pref=\prettyref

\newcommand{\diam}{\mathrm{diam}}

\ifnum\arxivmode=0
\renewcommand{\Ind}{\vvmathbb{1}}
\else
\renewcommand{\Ind}{\bm{1}}
\fi
\ifnum\jamesmode=0
\fi
\let\1\Ind

\newcommand\f{\varphi}

\newcommand*{\cut}{\mathsf{cut}}

\ifnum\arxivmode=0 \newcommand{\Indc}{\vvmathbb{I}} \else \newcommand{\Indc}{\mathbb{I}} \fi

\begin{document}

\title{Separators in region intersection graphs}

\author{James R. Lee\footnote{University of Washington.  Partially supported by NSF CCF-1407779.}} \date{}

\maketitle

\begin{abstract} For undirected graphs $G=(V,E)$ and $G_0=(V_0,E_0)$, say that $G$ is a {\em region intersection graph over $G_0$} if there is a family of connected subsets $\{ R_u \subseteq V_0 : u \in V \}$  of $G_0$ such that $\{u,v\} \in E \iff R_u \cap R_v \neq \emptyset$.

We show if $G_0$ excludes the complete graph $K_h$ as a minor for some $h \geq 1$, then every region intersection graph $G$ over $G_0$ with $m$ edges has a balanced separator with at most $c_h \sqrt{m}$ nodes, where $c_h$ is a constant depending only on $h$. If $G$ additionally has uniformly bounded vertex degrees, then such a separator is found by spectral partitioning.

A string graph is the intersection graph of continuous arcs in the plane.
String graphs are precisely region intersection graphs over planar graphs.
Thus the preceding result implies that every string graph with $m$ edges has a balanced separator of size $O(\sqrt{m})$. This bound is optimal, as it generalizes the planar separator theorem. It confirms a conjecture of Fox and Pach (2010), and improves over the $O(\sqrt{m} \log m)$ bound of Matou{\v{s}}ek (2013). \end{abstract}

\tableofcontents

\section{Introduction}

Consider an undirected graph $G_0=(V_0,E_0)$. A graph $G=(V,E)$ is said to be a {\em region intersection graph (rig) over $G_0$} if the vertices of $G$ correspond to connected subsets of $G_0$ and there is an edge between two vertices of $G$ precisely when those subsets intersect. Concretely, there is a family of connected subsets $\{R_u \subseteq V_0 : u \in V\}$ such that $\{u,v\} \in E \iff R_u \cap R_v \neq \emptyset$. For succinctness, we will often refer to $G$ as a {\em rig over $G_0$.}

Let $\rig(G_0)$ denote the family of all finite rigs over $G_0$. Prominent examples of such graphs include the intersection graphs of pathwise-connected regions on a surface (which are intersection graphs over graphs that can be drawn on that surface).

For instance, {\em string graphs} are the intersection graphs of continuous arcs in the plane. It is easy to see that every finite string graph $G$ is a rig over some planar graph: By a simple compactness argument, we may assume that every two strings intersect a finite number of times.  Now consider the planar graph $G_0$ whose vertices lie at the intersection points of strings and with edges between two vertices that are adjacent on a string (see \pref{fig:strings}). Then $G \in \rig(G_0)$.
It is not too difficult to see that the converse is also true; see
\pref{lem:planar-rigs}.

To illustrate the non-trivial nature of such objects, we recall that there are string graphs on $n$ strings that require $2^{\Omega(n)}$ intersections in any such representation \cite{KM91}. The recognition problem for string graphs is NP-hard \cite{Kra91}. Decidability of the recognition problem was established in \cite{SS04}, and membership in NP was proved in \cite{SSS03}. We refer to the recent survey \cite{Mat15} for more of the background and history behind string graphs.

Even when $G_0$ is planar, the rigs over $G_0$ can be dense:  Every complete graph is a rig over some planar graph (in particular, every complete graph is a string graph). It has been conjectured by Fox and Pach \cite{FP10} that every $m$-edge string graph has a balanced separator with $O(\sqrt{m})$ nodes. Fox and Pach proved that such graphs have separators of size $O(m^{3/4} \sqrt{\smash[b]{\log m}})$ and presented a number of applications of their separator theorem. Matou{\v{s}}ek \cite{Mat14} obtained a near-optimal bound of $O(\sqrt{m} \log m)$.
In the present work, we confirm the conjecture of Fox and Pach, and generalize the result to include all rigs over graphs that exclude a fixed minor.

\begin{theorem}\label{thm:main} If $G \in \rig(G_0)$ and $G_0$ excludes $K_h$ as a minor, then $G$ has a $\frac23$-balanced separator of size at most $c_h \sqrt{m}$ where $m$ is the number of edges in $G$. Moreover, one has the estimate $c_h \leq O(h^3 \sqrt{\smash[b]\log\, h})$. \end{theorem}

In the preceding statement, an $\epsilon$-balanced separator of $G=(V,E)$ is a subset $S \subseteq V$ such that in the induced graph $G[V\setminus S]$, every connected component contains at most $\epsilon |V|$ vertices.

The proof of \pref{thm:main} is constructive, as it is based on solving and rounding a linear program; it yields a polynomial-time algorithm for constructing the claimed separator. In the case when there is a bound on the maximum degree of $G$, one can use the well-known spectral bisection algorithm (see \pref{sec:l2conformal}).

Since planar graphs exclude $K_5$ as a minor, \pref{thm:main} implies that $m$-edge string graphs have $O(\sqrt{m})$-node balanced separators. Since the graphs that can be drawn on any compact surface of genus $g$ exclude a $K_h$ minor for $h \leq O(\sqrt{g+1})$, \pref{thm:main} also applies to string graphs over any fixed compact surface.

In addition, it implies the Alon-Seymour-Thomas \cite{AST90} separator theorem\footnote{Note that \pref{thm:main} is quantitatively weaker in the sense that \cite{AST90} shows the existence of separators with $O(h^{3/2} \sqrt{n})$ vertices.  Since every $K_h$-minor-free graph has at most $O(nh\sqrt{\smash[b]\log\,h})$ edges \cite{Kost82,Thom84}, our bound is $O(h^{7/2} (\log h)^{3/4} \sqrt{n})$.} for graphs excluding a fixed minor, for the following reason. Let us define the {\em subdivision of a graph $G$} to be the graph $\dot{G}$ obtained by subdividing every edge of $G$ into a path of length two. Then every graph $G$ is a rig over $\dot{G}$, and it is not hard to see that for $h \geq 1$, $G$ has a $K_h$ minor if and only if $\dot{G}$ has a $K_h$ minor.

\begin{figure} \begin{center} \includegraphics[width=11cm]{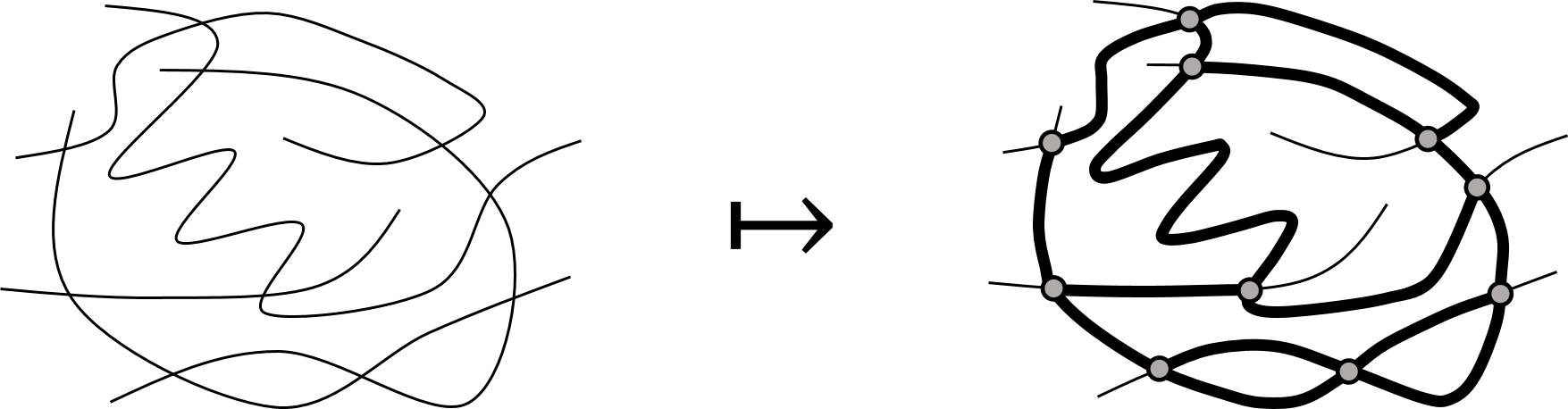} \caption{A string graph as a rig over a planar graph.\label{fig:strings}} \end{center} \end{figure}

\bigskip

\paragraph{Applications in topological graph theory}

We mention two applications of \pref{thm:main} in graph theory. In \cite{FP14}, the authors present some applications of separator theorems for string graphs.  In two cases, the tight bound for separators leads to tight bounds for other problems. The next two theorems confirm conjectures of Fox and Pach; as proved in \cite{FP14}, they follow from \pref{thm:main}. Both results are tight up to a constant factor.

\begin{theorem} There is a constant $c > 0$ such that for every $t \geq 1$, it holds that every $K_{t,t}$-free string graph on $n$ vertices has at most $c n t(\log t)$ edges. \end{theorem}

A {\em topological graph} is a graph drawn in the plane so that its vertices are represented by points and its edges by curves connecting the corresponding pairs of points.

\begin{theorem}\label{thm:bi-crossing} In every topological graph with $n$ vertices and $m \geq 4n$ edges, there are two disjoint sets, each of cardinality \begin{equation}\label{eq:bi-crossing} \Omega\left(\frac{m^2}{n^2 \log \frac{n}{m}}\right) \end{equation} so that every edge in one set crosses all edges in the other. \end{theorem}

This improves over the bound of $\Omega\left(\frac{m^2}{n^2 (\log \frac{n}{m})^c}\right)$ for some $c > 0$ proved in \cite{FPT10}, where the authors also show that the bound \eqref{eq:bi-crossing} is tight.
Before we conclude this section, let us justify the observation made earlier.

\begin{lemma}
\label{lem:planar-rigs}
Finite string graphs are precisely finite region intersection graphs over planar graphs.
\end{lemma}

\begin{proof}
We have already argued that string graphs are planar rigs.
Consider now a planar graph $G_0=(V_0,E_0)$ and a finite graph $G=(V,E)$ such that $G \in \rig(G_0)$.  Let $\{ R_u \subseteq V_0 : u \in V\}$
be a representation of $G$ as a rig over $G_0$.

Since $G$ is finite, we may assume that each region $R_u$ is finite.
To see this, for $v \in V_0$, let its {\em type} be the set $T(v) = \{ u \in V : v \in R_u \}$.  Then since $G$ is finite, there are only finitely many types. For any region $R_u \subseteq V_0$, let $\tilde R_u$ be a finite set of vertices
that exhausts every type in $R_u$, and let $\hat R_u$ be a finite spanning tree of $\tilde R_u$ in the induced graph $G_0[R_u]$.  Then the regions $\{\hat R_u : u \in V\}$
are finite and connected, and also form a representation of $G$ as a rig over $G_0$.

When each region $R_u$ is finite, we may assume also that $G_0$ is finite.
Now take a planar drawing of $G_0$ in $\R^2$ where the edges
 of $G_0$ are drawn as continuous arcs, and for every $u \in V$, let
$T_u \subseteq \R^2$ be the drawing of the spanning tree of $R_u$.
Each $T_u$ can be represented by a string (simply trace the tree using
an in-order traversal that begins and ends at some fixed node), and thus
$G$ is a string graph.
\end{proof}

\subsection{Balanced separators and extremal spread} \label{sec:intro-spread}

Since complete graphs are string graphs, we do not have access to topological methods based on the exclusion of minors.  Instead, we highlight a more delicate structural theory. The following fact is an exercise.

\begin{fact*} If $\dot{G}$ is a string graph, then $G$ is planar. \end{fact*}

More generally, we recall that {\em $H$ is a minor of $G$} if $H$ can be obtained from $G$ by a sequence of edge contractions, edge deletions, and vertex deletions. If $H$ can be obtained using only edge contractions and vertex deletions, we say that {\em $H$ is a strict minor of $G$.} The following lemma appears in \pref{sec:careful}.

\begin{lemma}\label{lem:strict-minor} If $G \in \rig(G_0)$ and $\dot{H}$ is a strict minor of $G$, then $H$ is a minor of $G_0$. \end{lemma}

This topological structure of (forbidden) strict minors in $G$ interacts nicely with ``conformal geometry'' on $G$,
as we now explain.
Consider the family of all pseudo-metric spaces that arise from a finite graph $G$ by assigning non-negative lengths to its edges and taking the induced shortest path distance. Certainly if we add an edge to $G$, the family of such spaces can only grow (since by giving the edge length equal to the diameter of the space, we effectively remove it from consideration). In particular, if $G=K_n$ is the complete graph on $n$ vertices, then every $n$-point metric space is a path metric on $G$.

The same phenomenon does not arise when one instead considers {\em vertex-weighted} path metrics on $G$. A {\em conformal graph} is a pair $(G,\omega)$ such that $G=(V,E)$ is a graph, and $\omega : V \to \R_+$. This defines a pseudo-metric as follows:  Assign to every $\{u,v\} \in E$ a length equal to $\frac{\omega(u)+\omega(v)}{2}$ and let $\dist_{\omega}$ be the induced shortest path distance. We will refer to $\omega$ as a {\em conformal metric on $G$} (and sometimes we abuse terminology and refer to $\dist_{\omega}$ as a conformal metric as well).

A significant tool will be the study of extremal conformal metrics on a graph $G$. Unlike in the edge-weighted case, the family of path distances coming from conformal metrics can be well-behaved even if $G$ contains arbitrarily large complete graph minors. As a simple example, let $K_{\N}$ denote the complete graph on countably many vertices.
Every countable metric space is a shortest-path metric on some edge-weighting of $K_{\N}$, and yet
every distance arising from a conformal metric $\omega : V_{K_{\N}} \to \R_+$ is bi-Lipschitz to the ultrametric
$(u,v) \mapsto \max\{\omega(u),\omega(v)\}$.

\bigskip

\paragraph{Vertex expansion and observable spread}

Fix a graph $G=(V_G,E_G) \in \rig(G_0)$ with $n=|V_G|$ and $m=|E_G|$. Since the family $\rig(G_0)$ is closed under taking induced subgraphs, a standard reduction allows us to focus on finding a subset $U \subseteq V_{G}$ with small isoperimetric ratio: $\frac{|\partial U|}{|U|} \lesssim \frac{\sqrt{m}}{n}$, where \[ \partial U \seteq \{ v \in U : E_G(v,V_G \setminus U) \neq \emptyset\}\,, \] and $E_G(v,V_G \setminus U)$ is the set of edges between $v$ and vertices outside $U$. Also define the interior $U^{\circ} = U \setminus \partial U$.

Let us define the {\em vertex expansion constant} of $G$ as \begin{equation}\label{eq:expansion} \phi_G \seteq \min \left\{ \frac{|\partial U|}{|U|} : \emptyset \neq U \subseteq V_G, |U^{\circ}| \leq \frac{|V_G|}{2} \right\}\,. \end{equation}

In \cite{FHL08}, it is shown that this quantity is related to the concentration of Lipschitz functions
on extremal conformal metrics on $G$.
(The study of such properties has a rich history; consider, for instance,
the concentration function in the sense of L\'evy and Milman (e.g., \cite{MS86}) and 
Gromov's observable diameter \cite{Gromov07}.)

For a finite metric space $(X,\dist)$, we define the {\em spread of $X$} as the quantity \[\spread(X,\dist) \seteq \frac{1}{|X|^2} \sum_{x,y \in X} \dist(x,y)\,.\]
Define the {\em observable spread of $X$} by \begin{equation}\label{eq:sobs} \sobs(X,\dist) \seteq \sup_{f : X \to \R} \left\{ \frac{1}{|X|^2} \sum_{x,y \in X} |f(x)-f(y)| : \textrm{$f$ is $1$-Lipschitz}\right\}\,. \end{equation}

\begin{remark}
 We remark on the terminology:  In general, it is difficult to ``view'' a large metric space all at once; this holds both conceptually and from an algorithmic standpoint. If one thinks of Lipschitz maps $f : X \to \R$ as ``observations'' then the observable spread captures how much of the spread can be ``seen.''
\end{remark}

We then define the {\em $L^1$-extremal observable spread of $G$} as \begin{equation}\label{eq:csobs} \csobs(G) \seteq \sup_{\omega : V_G \to \R_+} \left\{ \sobs(V_G,\dist_{\omega}) : \|\omega\|_{L^1(V_G)} \leq 1\right\}\,, \end{equation} where $\|\omega\|_{L^1(V_G)} \seteq \frac{1}{|V_G|} \sum_{v \in V_G} \omega(v)$.

Such extremal quantities arise naturally in the study of linear programming relaxations
of discrete optimization problems (like finding the smallest balanced vertex separator in a graph).
Related extremal notions are often employed
in conformal geometry and its discretizations;
see, in particular, the notions of extremal length
employed by Duffin \cite{Duffin62} and 
Cannon \cite{Cannon94}.
   
   In \pref{sec:conformal-width}, we recall the proof of the following theorem from \cite{FHL08} that relates expansion to the observable spread.

\begin{theorem}[\cite{FHL08}] \label{thm:spread} For every finite graph $G$, \[ \tfrac12 \csobs(G) \leq \frac{1}{\phi_G} \leq 3 \csobs(G)\,. \] \end{theorem}

\begin{example} If $G$ is the subgraph of the lattice $\Z^d$ on the vertex set $\{0,1,\ldots,L\}^{d}$, then $\phi_G \asymp 1/L$ and $\cspread(G) \asymp L$.  This can be achieved by taking $\omega \equiv 1$ and defining $f : V_G \to \R$ by $f(x)=x_1$. \end{example}

In light of \pref{thm:spread}, to prove \pref{thm:main}, it suffices to give a lower bound on $\csobs(G)$. It is
natural to compare this quantity to the {\em $L^1$-extremal spread of $G$}: \begin{equation}\label{eq:cspread} \cspread(G) \seteq \max \left\{ \frac{1}{|V_G|^2} \sum_{u,v \in V_G} \dist_{\omega}(u,v) : \|\omega\|_{L^1(V_G)} \leq 1\right\}\,. \end{equation}
Let us examine these two notions for planar graphs
using the theory of circle packings.

\begin{example}[Circle packings] Suppose that $G$ is a finite planar graph. The Koebe-Andreev-Thurston circle packing theorem asserts that $G$ is the tangency graph of a family $\{ D_v : v \in V_G\}$ of circles on the unit sphere $\mathbb{S}^2 \subseteq \R^3$. Let $\{c_v : v\in V_G\} \subseteq \mathbb{S}^2$ and $\{r_v > 0 : v \in V_G\}$ be the centers and radii of the circles, respectively. An argument of Spielman and Teng \cite{spielman-teng} (see also Hersch \cite{Hersch70} for the analogous result for conformal mappings) shows that one can take $\sum_{v \in V_G} c_v = \bm{0}$.

If we define $\omega(v) = r_v$ for $v \in V_G$, then $\dist_{\omega} \geq \dist_{\mathbb{S}^2} \geq \dist_{\R^3}$ on the centers $\{c_v : v \in V_G\}$.  (The latter two distances are the geodesic distance on $\mathbb{S}^2$ and the Euclidean distance on $\R^3$, respectively).

Using the fact that $\sum_{v \in V_G} c_v = \bm{0}$, we have \begin{equation}\label{eq:euclidean} \sum_{u,v \in V_G} \|c_u-c_v\|_2^2 = 2 n \sum_{u \in V_G} \|c_v\|^2 = 2 n^2\,. \end{equation} This yields \[ \sum_{u,v \in V_G} \dist_{\omega}(u,v) \geq \sum_{u,v\in V_G} \|c_u-c_v\| \geq \frac{n^2}{2}\,. \] Moreover,
 \[ \|\omega\|_{L^1(V_G)} \leq \|\omega\|_{L^2(V_G)} = \sqrt{\frac{1}{n} \sum_{v \in V_G} r_v^2} \leq \sqrt{\frac{\vol(\mathbb{S}^2)}{\pi n}} = \sqrt{\frac{4}{n}}\,.
 \] It follows that $\cspread(G) \geq \frac{\sqrt{n}}{4}$.

Observe that the three coordinate projections $\R^3 \to \R$ are all Lipschitz with respect to $\dist_{\omega}$, and one of them contributes at least a $1/3$ fraction to the sum \eqref{eq:euclidean}. We conclude that $\csobs(G) \geq \frac{\sqrt{n}}{12}$.
Combined with \pref{thm:spread}, this yields a proof of the Lipton-Tarjan separator theorem \cite{LT79}.
Similar proofs of the separator theorem based on circle packings are known (see \cite{MTWV97}),
and this one is not new (certainly it was known to the authors of \cite{spielman-teng}).
\end{example}

We will prove \pref{thm:main} in two steps:  By first giving a lower bound $\cspread(G) \gtrsim n/\sqrt{m}$ and then establishing $\csobs(G) \gtrsim \cspread(G)$.

\medskip

For the first step, we follow \cite{Mat14,FHL08,BLR08}.  The optimization \eqref{eq:cspread} is a linear program, and the dual optimization is a maximum multi-flow problem in $G$.
(See \pref{sec:flows} for a detailed discussion and statement of the duality theory.)
\matousek shows that
if $\hat{G}$ is a string graph,
then a multi-flow with small vertex congestion 
in $\hat{G}$
can be used to construct a multi-flow in a related
planar graph $G$ that has low vertex
congestion {\em in the $\ell_2$ sense.}
This element of the proof is crucial and ingenious; the reduction from a string graph $\hat{G}$
to a planar graph $G$ does not preserve congestion in the more standard $\ell_{\infty}$ sense.

Our work in \cite{BLR08} shows that such a multi-flow with small $\ell_2$ congestion cannot exist,
and thus one
concludes that there is no low-congestion flow in $\hat{G}$, providing a lower bound on $\cspread(G)$ via LP duality.
In \pref{sec:flows}, we extend this argument to rigs over $K_h$-minor-free graphs using the flow crossing framework of \cite{BLR08}.

\bigskip

\paragraph{Spread vs. observable spread}

Our major departure from \cite{Mat14} comes in the second step: Rounding a fractional separator to an integral separator by establishing that $\csobs(G) \geq C_h \cdot \cspread(G)$ when $G$ is a rig over a $K_h$-minor-free graph. \matousek used the following result that holds for any finite metric space.
It follows easily from the arguments of \cite{Bourgain85} or \cite{LR99} (see also \cite[Ch. 15]{Mat01}).

\begin{theorem} For any finite metric space $(X,d)$ with $|X| \geq 2$, it holds that \[ \sobs(X,d) \geq \frac{\spread(X,d)}{O(\log |X|)}\,. \] In particular, for any graph $G$ on $n \geq 2$ vertices, \[ \csobs(G) \geq \frac{\cspread(G)}{O(\log n)}\,. \] \end{theorem}

Instead of using the preceding result, we employ the graph partitioning method of Klein, Plotkin, and Rao \cite{KPR93}. Those authors present an iterative process for repeatedly partitioning a metric graph $G$ until the diameter of the remaining components is bounded.  If the partitioning process fails, they construct a $K_h$ minor in $G$.

Since rigs over $K_h$-minor-free graphs do not necessarily exclude any minors, we need to construct a different sort of forbidden structure. This is the role that \pref{lem:strict-minor} plays in \pref{sec:careful}. In order for the argument to work, it is essential that we construct induced partitions:  We remove a subset of the vertices which induces a partitioning of the remainder into connected components.

After constructing a suitable random partition of $G$, standard methods from metric embedding theory allow us to conclude in \pref{thm:embeddings} that if $G$ is a rig over some $K_h$-minor-free graph, then \[ \csobs(G) \geq \frac{\cspread(G)}{O(h^2)}\,. \]

\subsection{Eigenvalues and $L^2$-extremal spread} \label{sec:l2conformal}

In \pref{sec:eigenvalues}, we show how the methods presented here can be used to control eigenvalues of the discrete Laplacian on rigs.  Consider the linear space $\R^{V_G} = \{f : V_G \to \R\}$. Let $\cL_G : \R^{V_G} \to \R^{V_G}$ be the symmetric, positive semi-definite linear operator given by \[ \cL_G f(v) = \sum_{u : \{u,v\} \in E_G} (f(v)-f(u))\,. \] Let $0=\lambda_0(G) \leq \lambda_1(G) \leq \cdots \leq \lambda_{|V_G|-1}(G)$ denote the spectrum of $\cL_G$.

Define the {\em $L^p$-extremal spread} of $G$ as \begin{equation}\label{eq:Lpspread} \cspread_p(G) = \max_{\omega : V_{G} \to \R_+} \left\{ \frac{1}{|V_G|^2} \sum_{u,v \in V_G} \dist_{\omega}(u,v) : \|\omega\|_{L^p(V_G)}\leq 1\right\}\,. \end{equation}

In \cite{BLR08}, the $L^2$-extremal spread is used to give upper bounds on the first non-trivial eigenvalue of graphs that exclude a fixed minor. In \cite{KLPT09}, a stronger property of conformal metrics is used to bound the higher eigenvalues as well.  Roughly speaking, to control the $k$th eigenvalue, one requires a conformal metric $\omega : V_G \to \R_+$ such that the spread on every subset of size $\geq |V_G|/k$ is large. Combining their main theorems with the methods of \pref{sec:separators} and \pref{sec:flows}, we prove the following theorem in \pref{sec:eigenvalues}.

\begin{theorem}\label{thm:eigenvalues} Suppose that $G \in \rig(G_0)$ and $G_0$ excludes $K_h$ as a minor for some $h \geq 3$. If $d_{\max}$ is the maximum degree of $G$, then for any $k=1,2,\ldots,|V_G|-1$, it holds that \[ \lambda_k(G) \leq O(d_{\max}^2 h^6 \log h) \frac{k}{|V_G|}\,. \] \end{theorem}

In particular, the bound on $\lambda_1(G)$ shows that if $d_{\max}(G) \leq O(1)$, then recursive spectral partitioning (see \cite{spielman-teng}) finds an $O(\sqrt{n})$-vertex balanced separator in $G$.

\subsection{Additional applications} \label{sec:addl}

\medskip \noindent {\bf Treewidth approximations.} Bounding $\csobs(G)$ for rigs over $K_h$-minor-free graphs leads to some additional applications.  Combined with the rounding algorithm implicit in \pref{thm:spread} (and explicit in \cite{FHL08}), this yields an $O(h^2)$-approximation algorithms for the vertex uniform Sparsest Cut problem. In particular, it follows that if $G \in \rig(G_0)$ and $G_0$ excludes $K_h$ as a minor, then there is a polynomial-time algorithm that constructs a tree decomposition of $G$ with treewidth $O(h^2 \mathrm{tw}(G))$, where $\mathrm{tw}(G)$ is the treewidth of $G$. This result appears new even for string graphs. We refer to \cite{FHL08}.

\medskip \noindent {\bf Lipschitz extension.} The padded decomposability result of \pref{sec:pad} combines with the Lipschitz extension theory of \cite{LN05} to show the following. Suppose that $(G,\omega)$ is a conformal graph, where $G$ is a rig over some $K_h$-minor free graph. Then for every Banach space $Z$, subset $S \subseteq V_G$, and $L$-Lipschitz mapping $f : S \to Z$, there is an $O(h^2 L)$-Lipschitz extension $\tilde f : V_G \to Z$ with $\tilde f|_{S} = f$.  See \cite{MM16} for applications to flow and cut sparsifiers in such graphs.

\subsection{Preliminaries}

We use the notation $\R_+ = [0,\infty)$ and $\Z_+ = \Z \cap \R_+$.

All graphs appearing in the paper are finite and undirected unless stated otherwise. If $G$ is a graph, we use $V_G$ and $E_G$ for its edge and vertex sets, respectively. If $S \subseteq V_G$, then $G[S]$ is the induced subgraph on $S$. For $A,B \subseteq V_G$, we use the notation $E_G(A,B)$ for the set of all edges with one endpoint in $A$ and the other in $B$. Let $N_G(A) = A \cup \partial A$ denote the neighborhood of $A$ in $G$. We write $\dot{G}$ for the graph that arises from $G$ by subdividing every edge of $G$ into a path of length two.

If $(X,\dist)$ is a pseudo-metric space and $S,T \subseteq X$, we write $\dist(x,S) = \inf_{y \in S} \dist(x,y)$ and $\dist(S,T) = \inf_{x \in S, y\in T} \dist(x,y)$.

Finally, we employ the notation $A \lesssim B$ to denote $A \leq O(B)$, which means there exists a universal (unspecified) constant $C > 0$ for which $A \leq C \cdot B$.

\section{Vertex separators and conformal graph metrics} \label{sec:separators}

The following result is standard.  Recall the definition of the vertex expansion \eqref{eq:expansion}.

\begin{lemma}\label{lem:chomp} Suppose that every induced subgraph $H$ of $G$ satisfies $\phi_H \leq \frac{\alpha}{|V_H|}$. Then $G$ has a $\frac23$-balanced separator of size at most $\alpha$. \end{lemma}

Thus in the remainder of this section, we focus on bounding $\phi_G$. We remark on one other basic fact:  Consider a graph $G$, a partition $V_G = A \cup B$ and a subset $S \subseteq V$.  If $E_G(A \setminus S, B \setminus S)=\emptyset$, then \begin{equation}\label{eq:conductance} |S| \geq \phi_G \frac{|A| \cdot |B|}{|V_G|}\,. \end{equation}
Indeed, suppose that $|A| \leq |B|$.  Then, $\phi_G \leq \frac{|S|}{|A \cup S|} \leq \frac{|S|}{|A|} \frac{|V_G|}{|B|}\,.$

\subsection{Conformal graphs} \label{sec:conformal-width}

A conformal graph is a pair $(G,\omega)$ where $G$ is a connected graph and $\omega : V_G \to \R_+$. Associated to $(G,\omega)$, we define a distance function $\dist_{\omega}$ as follows. Assign a length $\frac{\omega(u)+\omega(v)}{2}$ to every $\{u,v\} \in E_G$; then $\dist_{\omega}$ is the induced shortest-path metric.
For $U \subseteq V_G$, we define $\diam_{\omega}(U) = \sup_{u,v \in U} \dist_{\omega}(u,v)$.

The extremal $L^1$-spread is a linear programming relaxation of the optimization \eqref{eq:expansion} defining $\phi_G$
(up to universal constant factors). In \pref{sec:vcon-rigs}, we establish the following result.

\begin{theorem}\label{thm:small-width} If $\hat G$ is a connected graph and $\hat G \in \rig(G)$ for some graph $G$ that excludes a $K_h$ minor, then \begin{equation}\label{eq:big-spread} \cspread_1(\hat G) \gtrsim \frac{n}{h\sqrt{m\log h}}\,, \end{equation} where $n=|V_{\hat G}|$ and $m=|E_{\hat G}|$. \end{theorem}

Recall that \pref{thm:small-width} completes the first step of our program for exhibiting small separators.  For the second step, we need to relate $\cspread_1(G)$ to $\csobs(G)$.  Before that, we restate the proof of \pref{thm:spread} from \cite{FHL08} in our language. (The proof presented below is also somewhat simpler, as it does not employ Menger's theorem as in \cite{FHL08}.)

\begin{theorem}[Restatement of \pref{thm:spread}] \label{thm:fhl} For any connected graph $G$, it holds that \begin{equation}\label{eq:fhl08} \tfrac12 \csobs(G) \leq \frac{1}{\phi_G} \leq 3 \, \csobs(G)\,. \end{equation} \end{theorem}

\begin{proof} The right-hand inequality is straightforward: Suppose that $U \subseteq V_G$ witnesses $\phi_G = |\partial U|/|U|$. Let $\bar{U} = V_G \setminus U$ and $U^{\circ} = U \setminus \partial U$. Let $n=|V_G|$ and $s=|\partial U|$. Partition $\partial U = S_1 \cup S_2$ so that $|S_1|=\lceil s/2\rceil$ and $S_2 = \lfloor s/2\rfloor$.

Now let $\omega = \1_{\partial U}$ and
 define two maps
\begin{align*} f_1(v) = \begin{cases} \frac{-1}{2} & v \in U^{\circ}\\ 0 & v \in \partial U \\ \frac{1}{2} & v \in \bar{U} \end{cases}  \qquad\qquad f_2(v) = \begin{cases} \frac{-1}{2} & v \in S_1 \\ 0 & v \notin \partial U \\ \frac{1}{2} & v \in S_2\,. \end{cases} \end{align*}

Since $\partial U$ separates $U^{\circ}$ from $\bar{U}$, the maps $f_1, f_2 : (V_G, \dist_{\omega}) \to \R$ are $1$-Lipschitz, and \begin{align*} \sum_{u,v \in V_G} |f_1(u)-f_1(v)| &= 2|U^{\circ}| \cdot |\bar{U}| + |\partial U|(|U^{\circ}|+|\bar{U}|) \\ \sum_{u,v \in V_G} |f_2(u)-f_2(v)| &= 2 |S_1| \cdot |S_2| + |\partial U| \left(|U^{\circ}| + |\bar{U}|\right)\,. \end{align*}
 If $|\partial U| \geq |U^{\circ}|$, then the map $f_2$ satisfies $\frac{1}{|V_G|^2} \sum_{x,y \in V_G} |f_2(x)-f_2(y)| \geq \frac{1}{3\phi_G}.$ Otherwise, the map $f_1$ satisfies $\frac{1}{|V_G|^2} \sum_{x,y \in V_G} |f_1(x)-f_1(y)| \geq \frac{1}{2\phi_G}.$ In either case, we have shown that $\csobs(G) \geq \frac{1}{3\phi_G}$.

\medskip \paragraph{Left-hand inequality} Now we establish the more interesting bound. Suppose that $(G,\omega)$ is a conformal metric with $\|\omega\|_{L^1(V_G)}\leq 1$ and $f : (V_G,\dist_{\omega}) \to \R$ is a $1$-Lipschitz mapping with \[ \frac{1}{|V_G|^2} \sum_{x,y \in V_G} |f(x)-f(y)| = \csobs(G)\,. \]

For $\theta \in \R$, define the three sets \begin{align*} A_{\theta} &= \left\{ v \in V_G : f(v) \leq \theta \right\} \\ S_{\theta} &= \left\{ v \in V_G : |f(v)-\theta| \leq \frac12 \omega(v)\right\} \\ B_{\theta} &= \left\{ v \in V_G : f(v) > \theta \right\}\,. \end{align*} Observe that if $u \in A_{\theta} \setminus S_{\theta}$ and $v \in B_{\theta} \setminus S_{\theta}$, then $f(u) \leq \theta - \frac12 \omega(u)$ and $f(v) > \theta + \frac12 \omega(v)$, therefore $|f(u)-f(v)| > \frac{\omega(u)+\omega(v)}{2}$.  But if $\{u,v\} \in E_G$, then since $f$ is $1$-Lipschitz, it should hold that $|f(u)-f(v)| \leq \dist_{\omega}(u,v) = \frac{\omega(u)+\omega(v)}{2}$. We conclude that $E_G(A_{\theta} \setminus S_{\theta}, B_{\theta} \setminus S_{\theta})=\emptyset$. Therefore \eqref{eq:conductance} yields \begin{align*} \int_{\R} |S_{\theta}|\,d\theta
\geq \frac{\phi_G}{|V_G|} \int_{\R} |A_{\theta}| \cdot |B_{\theta}|\,d\theta =\frac{\phi_G}{2|V_G|} \sum_{x,y \in V_G} |f(x)-f(y)| = \frac{\phi_G |V_G|}{2}\, \csobs(G)\,. \end{align*}

On the other hand, \begin{align*} \int_{\R} |S_{\theta}|\,d\theta \leq \sum_{v \in V_G} \int_{\R} \1_{B(f(v), \omega(v)/2)}(\theta) \,d\theta = \sum_{v \in V_G} \omega(v) = |V_G| \cdot \|\omega\|_{L^1(V_G)}\,, \end{align*} where we have used the notation $B(x,r) = \{ y \in \R : |x-y| \leq r\}$. \end{proof}

In the next section, we prove the following theorem (though the main technical arguments appear in \pref{sec:careful}).

\begin{theorem}\label{thm:embeddings} If $\hat G \in \rig(G)$ and $G$ excludes a $K_h$ minor, then \begin{equation}\label{eq:khline} \cspread_1(\hat G) \leq O(h^2)\ \csobs(\hat G)\,. \end{equation} \end{theorem}

Combining \pref{lem:chomp} with \pref{thm:small-width}, \pref{thm:fhl}, and \pref{thm:embeddings} yields a proof of \pref{thm:main}. Indeed, suppose that $\hat G \in \rig(G)$ and $G$ excludes a $K_h$ minor. Let $n=|V_{\hat G}|$ and $m=|E_{\hat G}|$. Then, \[ \phi_{\hat G} \stackrel{\eqref{eq:fhl08}}{\leq} \frac{2}{\csobs({\hat G})} \stackrel{\eqref{eq:khline}}{\lesssim} \frac{h^2}{\cspread_1(\hat G)} \stackrel{\eqref{eq:big-spread}}{\lesssim} \frac{h^3 \sqrt{\smash[b]{m \log\, h}}}{n}\,, \] completing the proof in light of \pref{lem:chomp}.

\subsection{Padded partitions and random separators} \label{sec:pad}

Let $(X,d)$ be a finite metric space.  For $x \in X$ and $R \geq 0$, define the closed ball \[ B_d(x,R) = \{ y \in X :d(x,y) \leq R \}\,. \] If $P$ is a partition of $X$ and $x \in X$, we write $P(x)$ for the set of $P$ containing $x$.

Say that a partition $P$ is $\Delta$-bounded if $S \in P \implies \diam(S) \leq \Delta$.

\begin{definition} A {\em random} partition $\bm{P}$ of $X$ is $(\alpha,\Delta)$-padded if it is almost surely $\Delta$-bounded and for every $x \in X$, \[ \Pr\left[B_d\left(x,\tfrac{\Delta}{\alpha}\right) \subseteq \bm{P}(x)\right] \geq \tfrac12\,. \] \end{definition}

The following result is essentially contained in \cite{Rab08} (see also \cite[Thm 4.4]{BLR08}). We recall the argument here since the exact statement we need has not appeared.

\begin{lemma}\label{lem:rab} Let $(X,d)$ be a finite metric space.
If $(X,d)$ admits an $(\alpha, \spread(X,d)/4)$-padded partition, then \[ \spread(X,d) \leq 16\alpha\cdot \sobs(X,d)\,. \] \end{lemma}

The proof breaks into two cases whose conjunction yields \pref{lem:rab}.

\begin{lemma}\label{lem:rab1} If there is an $x_0 \in X$ such that $|B_{d}(x_0, \spread(X,d)/4)| \geq \tfrac12 |X|$, then $\sobs(X,d) \geq \frac14 \spread(X,d)$. \end{lemma}

\begin{proof} Let $B=B_d(x_0, \spread(X,d)/4)$. Define $f : X \to \R$ by $f(x)= d(x,B)$. This map is $1$-Lipschitz, and moreover \begin{align*} \sum_{x,y \in X} |f(x)-f(y)| \geq \sum_{\substack{x \in B \\ y \notin B}} d(y,B) \geq \frac{|X|}{2} \sum_{y \in X} d(y,B) \geq \frac{|X|}{2} \sum_{y \in X} \left[d(y,x_0)-\frac{\spread(X,d)}{4}\right]. \end{align*} On the other hand, \begin{align*} \sum_{x,y \in X} d(x,y) \leq \sum_{x,y \in X} \left[d(x,x_0)+d(y,x_0)\right] = 2 |X| \sum_{y \in X} d(y, x_0)\,. \end{align*} Combining the two preceding inequalities yields \[ \sobs(X,d) \geq \frac{1}{|X|^2} \sum_{x,y \in X} |f(x)-f(y)| \geq -\frac{\spread(X,d)}{8} + \frac1{4|X|^2} \sum_{x,y\in X} d(x,y) = \frac{\spread(X,d)}{4}\,.\qedhere \] \end{proof}

\begin{lemma} If $|B_d(x, \spread(X,d)/4)| < \frac12 |X|$ for all $x \in X$, then $\sobs(X,d) \geq \frac{1}{16\alpha} \spread(X,d)$. \end{lemma}

\begin{proof} Let $\bm{P}$ be an $(\alpha,\spread(X,d)/4)$-padded partition. Let $\bm{\sigma} : \bm{P} \to \{0,1\}$ be a map chosen uniformly at random conditioned on $\bm{P}$. Define $\bm{S} = \left\{ x \in X : \bm{\sigma}(\bm{P}(x)) = 1 \right\}$ and $\bm{F} : X \to \R$ by $\bm{F}(x) = d(x,\bm{S})$. Note that $\bm{F}$ is almost surely $1$-Lipschitz. Moreover, observe that \begin{equation}\label{eq:far} B_d(x, \spread(X,d)/4\alpha) \subseteq \bm{P}(x)  \wedge \bm{\sigma}(x)=0 \implies d(x,\bm{S}) \geq \frac{\spread(X,d)}{4\alpha}\,. \end{equation}

Therefore if $x,y \in X$ satisfy $d(x,y) > \spread(X,d)/4$, then \eqref{eq:far} and independence yields \[ \E\left[|\bm{F}(x)-\bm{F}(y)|\right] \geq \Pr[\bm{\sigma}(\bm{P}(x)) \neq \bm{\sigma}(\bm{P}(y))]\cdot \frac12 \frac{\spread(X,d)}{4\alpha}\,. \] By assumption, $\left|\left\{ (x,y) : d(x,y) > \spread(X,d)/4 \right\}\right| \geq \frac12 |X|^2$, hence \[ \sobs(X,d) \geq \frac{1}{|X|^2} \sum_{x,y \in X} \E\left[|\bm{F}(x)-\bm{F}(y)|\right] \geq \frac{\spread(X,d)}{16\alpha}\,.\qedhere \] \end{proof}

In order to produce a padded partition, we will construct an auxiliary random object. Let $(G,\omega)$ be a conformal graph. Define the {\em skinny ball:}  For $c \in V_G$ and $R \geq 0$, \[ {\cB}_{\omega}(c,R) = \left\{ v \in V_G : \dist_{\omega}(c,v) < R - \frac12 \omega(v) \right\},\\ \]

Say that a random subset $\bm{S} \subseteq V$ is an {\em $(\alpha,\Delta)$-random separator} if the following two conditions hold: \begin{enumerate} \item For all $v \in V_G$ and $R \geq 0$, \[ \Pr[\cB_{\omega}(v,R) \cap \bm{S} = \emptyset] \geq 1- \alpha \frac{R}{\Delta}\,. \] \item Almost surely every connected component of $G[V\setminus \bm{S}]$ has diameter at most $\Delta$ (in the metric $\dist_{\omega}$). \end{enumerate}

\begin{lemma}\label{lem:padsep} If $(G,\omega)$ admits an $(\alpha,\Delta)$-random separator, then $(V_G, \dist_{\omega})$ admits an $(8\alpha,\Delta)$-padded partition. \end{lemma}

\begin{proof} The random partition $\bm{P}$ is defined by taking all the connected components of $G[V \setminus \bm{S}]$, along with the single sets $\{\{x\} : x \in \bm{S}\}$. The fact that $\bm{P}$ is almost surely $\Delta$-bounded is immediate.

Set $R = \frac{\Delta}{2\alpha}$ and observe that for every $v \in V_G$, $\cB_{\omega}(v,R) \cap \bm{S} = \emptyset \implies \cB_{\omega}(v,R) \subseteq \bm{P}(x)$, since $\cB_{\omega}(v,R)$ is a connected set in $G$.

Moreover, $B_{\dist_{\omega}}(v,R/4) \subseteq \cB_{\omega}(v,R)$. To see this, observe that \[\max_{x \in B_{\dist_{\omega}}(v,R/4) \setminus \{v\}} \omega(x) \leq \tfrac{R}{2} \,,\] and thus \[ x \in B_{\dist_{\omega}}(v,R/4) \implies \dist_{\omega}(v,x) < R - \tfrac12 \omega(x) \implies x \in \cB_{\omega}(v,R)\,. \] It follows that for every $v \in V_G$, we have \[ \Pr\left[B_{\dist_{\omega}}(v_G, \tfrac{\Delta}{8\alpha}) \subseteq \bm{P}(x)\right] \geq \tfrac12\,, \] completing the proof that $\bm{P}$ is $(8\alpha,\Delta)$-padded. \end{proof}

This following result is proved in \pref{sec:careful} (see \pref{cor:random-sep}).

\begin{theorem}\label{thm:pad} If $\hat G \in \rig(G)$ and $G$ excludes $K_h$ as a minor, then for every $\Delta > 0$, every conformal metric $(G,\omega)$ admits an $(\alpha,\Delta)$-random separator with $\alpha \leq O(h^2)$. \end{theorem}

We can now prove \pref{thm:embeddings}.

\begin{proof}[Proof of \pref{thm:embeddings}] Suppose that $\hat G \in \rig(G)$ and $G$ excludes a $K_h$ minor. Let $(\hat G,\omega)$ be a conformal metric such that $\spread(V_{\hat G},\dist_{\omega})=\cspread_1(\hat G)$. Combining \pref{thm:pad} and \pref{lem:padsep} shows that $(V_{\hat G}, \dist_{\hat G})$ admits an $(\alpha,\cspread_1(\hat G)/4)$-random separator for some $\alpha \leq O(h^2)$.  Now \pref{lem:rab} shows that $\csobs(\hat G) \geq \sobs(V_{\hat G},\dist_{\omega}) \gtrsim \frac{1}{h^2} \spread(V_{\hat G}, \dist_{\omega}) = \frac{1}{h^2} \cspread_1(\hat G)$, completing the proof. \end{proof}

\begin{remark} One advantage to introducing the auxiliary random separator $\bm{S}$ is that it can be used to directly relate $\phi_G$ and $\cspread(G)$ without going through padded partitions. Indeed, this can be done using the weaker property that for every $v \in V_G$, \[ \Pr[v \in \bm{S}] \leq \alpha \frac{\omega(v)}{\Delta}\,. \] (The stronger padding property has a number of additional applications; see \pref{sec:applications} and \pref{sec:addl}.)

We present the argument. Suppose that $(G,\omega)$ is a conformal graph with $\|\omega\|_{L^1(V_G)} \leq 1$ and let $\mathfrak{s} \seteq \spread(V_G,\dist_{\omega}).$ If there is some vertex $v_0 \in V_G$ for which $|B_{\dist_{\omega}}(v_0, \mathfrak{s}/4)| \geq \tfrac12 |V_G|$, then one can apply \pref{lem:rab1} and \pref{thm:fhl} to obtain \[ \phi_G \leq \frac{2}{\csobs(G)} \leq \frac{2}{\sobs(V_G,\dist_{\omega})} \leq \frac{8}{\mathfrak{s}}\,.\]

Suppose that no such $v_0$ exists.  In that case, every subset $U \subseteq V_G$ with $\diam_{\omega}(U) \leq \mathfrak{s}/4$ has $|U| < |V_G|/2$. Let $\bm{S}$ be an $(\alpha, \mathfrak{s}/4)$-random separator.  Then every connected component of $G[V \setminus \bm{S}]$ has at most $|V_G|/2$ vertices. (In particular, $\bm{S}$ is a $\frac23$-balanced separator with probability 1.) Therefore by linearity of expectation, \[ \phi_G \leq \frac{\E\left[|\bm{S}|\right]}{|V_G|/2} \leq \frac{8\alpha\sum_{v \in V_G} \omega(v)}{|V_G| \mathfrak{s}} = \alpha \frac{8 \|\omega\|_{L^1(V_G)}}{\mathfrak{s}} \leq \frac{8\alpha}{\mathfrak{s}}\,. \] \end{remark}

\section{Multi-flows, congestion, and crossings} \label{sec:flows}

Let $G$ be an undirected graph, and let $\cP_G$ denote the the set of all paths in $G$. Note that we allow length-$0$ paths consisting of a single vertex. For vertices $u,v \in V_G$, we use $\cP_G^{uv} \subseteq \cP_G$ for the subcollection of $u$-$v$ paths. A {\em multi-flow in $G$} is a map $\Lambda : \cP_G \to \R_+$.
We will use the terms ``flow'' and ``multi-flow'' interchangeably.

 Define the {\em congestion map} $c_{\Lambda} : V_G \to \R_+$ by \[ c_{\Lambda}(v) = \sum_{\gamma \in \cP_G : v \in \gamma} \Lambda(\gamma)\,. \] For $u,v \in V_G$, we denote the total flow sent between $u$ and $v$ by
\[
\Lambda[u,v] = \sum_{\gamma \in \cP_G^{uv}} \Lambda(\gamma)\,.
\]
For an undirected graph $H$, an {\em $H$-flow in $G$} is a pair $(\Lambda, \f)$ that satisfies the following conditions: \begin{enumerate} \item  $\Lambda$ is a flow in $G$ \item $\f : V_H \to V_G$ \item For every $u,v \in V_G$, \[ \Lambda[u,v] = \#\left\{\{x,y\} \in E_H : \{\f(x),\f(y)\}=\{u,v\} \right\}\,.
\] \end{enumerate} If the map $\f$ is injective, we say that $(\Lambda,\f)$ is {\em proper.} Say that $(\Lambda,\f)$ is {\em integral} if $\{ \Lambda(\gamma) : \gamma \in \cP_G \} \subseteq \Z_+$.

\subsection{Duality between conformal metrics and multi-flows}

For $p \in [1,\infty]$, define the {\em $\ell_p$-vertex congestion of $G$} by \[\vcong_{p}(G) = \min_{\Lambda} \left\{ \|c_{\Lambda}\|_{\ell_p(V_G)} : \Lambda[u,v] = 1 \ \ \forall \{u,v\} \in {V_G \choose 2} \right\}\,, \] where the minimum is over all flows in $G$.

The next theorem follows from the strong duality of convex optimization; see \cite[Thm. 2.2]{BLR08} which employs Slater's condition for strong duality (see, e.g., \cite[Ch. 5]{BV04}).

\begin{theorem}[Duality theorem]\label{thm:duality} For every $G$, it holds that if $(p,q)$ is a pair of dual exponents, then \[ \vcong_p(G) = |V_G|^{2-\frac{1}{q}} \cdot \cspread_q(G)\,. \] \end{theorem}

We will only require the case $p=\infty, q=1$, except in \pref{sec:eigenvalues} where the $p=q=2$ case is central.

\subsection{Crossing congestion and excluded minors}

Now we define the {\em crossing congestion} of a flow: If $(\Lambda,\f)$ is an $H$-flow in $G$, denote \[ \cross_G(\Lambda,\f) = \sum_{\substack{\{u,v\},\{u',v'\} \in E_H \\ |\{u,v,u',v'\}|=4}}\ \sum_{\gamma \in \cP_G^{\f(u) \f(v)}} \sum_{\gamma' \in \cP_G^{\f(u') \f(v')}}
 \Lambda(\gamma) \Lambda(\gamma')\1_{\{\gamma \cap \gamma' \neq \emptyset\}}\,.
\]
It provides a lower bound on the $\ell_2$ congestion of $\Lambda$, as clearly
\[
\sum_{v \in V_G} c_{\Lambda}(v)^2 = \sum_{v \in V_G}\ \sum_{\substack{\gamma, \gamma'\in \cP_G : \\ v \in \gamma \cap \gamma'}} \Lambda(\gamma) \Lambda(\gamma')
\geq \sum_{\gamma,\gamma' \in \cP_G} \Lambda(\gamma) \Lambda(\gamma')\1_{\{\gamma \cap \gamma' \neq \emptyset\}}
\geq \chi_G(\Lambda,\f)\,.
\]

Define \[ \cross_G^*(H) = \inf_{(\Lambda,\f)} \cross_G(\Lambda,\f)\,, \] where the infimum is over all $H$-flows in $G$. Define also \[ \cross_G^{\dag}(H) = \min_{(\Lambda,\f)} \cross_G(\Lambda,\f)\,, \] where the infimum is over all {\em integral} $H$-flows in $G$.  The next lemma offers a nice property of crossing congestion: The infimum is always achieved by integral flows.

\begin{lemma} For every graph $H$, it holds that \[ \cross^*_G(H) = \cross^{\dag}_G(H)\,. \] \end{lemma}

\begin{proof} Given any $H$-flow $(\Lambda,\f)$, define a random integral flow $\bm{\Lambda}^{\dag}$ as follows: For every edge $\{u,v\} \in E_H$, independently choose a path $\gamma \in \cP_G^{\f(u) \f(v)}$ with probability $\frac{\Lambda(\gamma)}{\Lambda[\f(u),\f(v)]}$ and let $\bm{\Lambda}^{\dag}(\gamma)$ be equal to the number of edges of $E_H$ that choose the path $\gamma$. (For all paths $\gamma$ not selected in such a manner, $\bm{\Lambda}^{\dag}(\gamma)=0$.) Independence and linearity of expectation yield $\E[\cross_G(\bm{\Lambda}^{\dag},\f)]=\cross_G(\Lambda,\f)$. \end{proof}

The next result relates the topology of a graph to crossing congestion; it appears as \cite[Lem 3.2]{BLR08}.

\begin{lemma} If $H$ is a bipartite graph with minimum degree 2 and $(\Lambda,\f)$ is an $H$-flow in $G$ with $\cross_G(\Lambda,\f)=0$, then $G$ has an $H$-minor. \end{lemma}

The preceding lemma allows one to use standard crossing number machinery to arrive at the following result (see \cite[Thm 3.9--3.10]{BLR08}).

\begin{theorem} \label{thm:blr} For every $h \geq 2$, the following holds: If $G$ excludes $K_h$ as a minor, then for any $N \geq 4h$, \[ \cross^*_G(K_{N}) \gtrsim \frac{N^4}{h^3}\,. \] Moreover, there is a constant $K > 0$ such that if $N \geq  K h \sqrt{\smash[b]\log\,h}$, then \[ \cross^*_G(K_{N}) \gtrsim \frac{N^4}{h^2 \log h}\,. \] \end{theorem}

\subsection{Vertex congestion in rigs} \label{sec:vcon-rigs}

We now generalize Matou{\v{s}}ek's argument to prove the following theorem.

\begin{theorem}\label{thm:mat} For any graph $G$ and $\hat G \in \rig(G)$, \[ \cross_{G}^*(K_{|V_{\hat G}|}) \leq (4 |E_{\hat G}|+|V_{\hat G}|)\, \vcong_{\infty}(\hat G)^2\,. \] \end{theorem}

Before moving to the proof, we state the main result of this section. It follows immediately from the conjunction of \pref{thm:mat} and \pref{thm:blr}. (One does not require a lower bound on $n$ as in \pref{thm:blr} because the bound $\vcong_{\infty}(\hat G) \geq |V_{\hat G}|$ always holds.)

\begin{corollary}\label{cor:congestion} Suppose $\hat G$ is a connected graph and $\hat G \in \rig(G)$ for some graph $G$ that excludes $K_h$ as a minor. If $n=|V_{\hat G}|$ and $m=|E_{\hat G}|$, then \[ \vcong_{\infty}(\hat G) \gtrsim \sqrt{\frac{\cross_G^*(K_n)}{m}} \gtrsim \frac{n^2}{h\sqrt{m \log h}}\,. \] In particular, \pref{thm:duality} yields \[ \cspread_{1}(\hat G) \gtrsim \frac{n}{h \sqrt{m \log h}}\,. \] \end{corollary}

\begin{proof}[Proof of \pref{thm:mat}] Let $\{R_u : u \in V_{\hat G}\}$ be a set of regions realizing $\hat G$ over $G$. For every path $\gamma$ in $\hat G$, we specify a path $\check{\gamma}$ in $G$. For each $v \in V_{\hat G}$, fix some distinguished vertex $\check{v} \in R_v$.

   Suppose that $\gamma = (v_1, v_2, \ldots, v_k)$. Let $\check{\gamma}$ be any path $\check{\gamma} = \check{\gamma}_1 \circ \check{\gamma}_2 \circ \cdots \circ \check{\gamma}_{k}$
 which starts at $\check{v_1}$, ends at $\check{v_k}$,
and where for each $i=1,\ldots,k$, the entire subpath $\check{\gamma}_i$ is contained in $R_{v_i}$.  This is possible because each $R_{v_i}$ is connected and $\{v_i,v_{i+1}\} \in E_{\hat G}$ implies that $R_{v_i}$ and $R_{v_{i+1}}$ share at least one vertex of $G$. We describe the path $\check{\gamma}$ as ``visiting'' the regions $R_{v_1}, R_{v_2}, \cdots, R_{v_k}$ in order.

Let $n=|V_{\hat G}|$ and $m=|E_{\hat G}|$. Let $(\Lambda, \f)$ be a proper $K_n$-flow in $\hat G$ achieving $\|c_{\Lambda}\|_{\infty} = \vcong_{\infty}(\hat G)$. The path mapping $\gamma \mapsto \check\gamma$ sends $(\Lambda,\f)$ to a (possibly improper) $K_n$-flow $(\check \Lambda, \check\f)$ in $G$. Establishing the following claim will complete the proof of \pref{thm:mat}.

\begin{claim}\label{claim:crossup} It holds that \begin{equation}\label{eq:crossup} \cross_G(\check\Lambda, \check\f) \leq \sum_{u \in V_{\hat G}} c_{\Lambda}(u)^2 + \sum_{\{u,v\} \in E_{\hat G}} \left(c_{\Lambda}(u) + c_{\Lambda}(v)\right)^2 \leq (4m+n) \|c_{\Lambda}\|^2_{\infty}\,. \end{equation} \end{claim}

We prove the claim as follows:  If $\check{\gamma_1}$ and $\check{\gamma_2}$ intersect in $G$, we charge weight $\Lambda({\gamma_1}) \Lambda({\gamma_2})$ to some element in $V_{\hat G} \cup E_{\hat G}$. If $\check{\gamma_1}$ visits the regions $R_{u_1}, R_{u_2}, \cdots, R_{u_{k_1}}$ and $\check{\gamma_2}$ visits the regions $R_{v_1}, R_{v_2}, \cdots, R_{v_{k_2}}$ and $\check{\gamma_1} \cap \check{\gamma_2} \neq \emptyset$, then they meet at some vertex $x \in R_{u_i} \cap R_{v_j}$.  If $u_i=v_j$, we charge this crossing to $u_i \in V_{\hat G}$.  Otherwise we charge this crossing to the edge $\{u_i,v_j\} \in E_{\hat G}$.

If $u \in V_{\hat G}$ is charged by $(\gamma_1,\gamma_2)$, then $u \in \gamma_1 \cap \gamma_2$. Thus the total weight charged to $u$ is at most \[ \sum_{\gamma \neq \gamma' \in \cP_{\hat G} : u \in \gamma} \Lambda(\gamma) \Lambda(\gamma') \leq c_{\Lambda}(u)^2\,. \] Similarly, if $\{u,v\} \in E_{\hat G}$ is charged by $(\gamma_1, \gamma_2)$, then $\{u,v\} \subseteq \gamma_1 \cup \gamma_2$, thus the total weight charged to $\{u,v\}$ at most $\left(c_{\Lambda}(u)+c_{\Lambda}(v)\right)^2$. Since all of the weight contributing to $\chi_{G}(\check{\Lambda},\check{\f})$ has been charged, this yields the desired claim. \qedhere \end{proof}

\section{Careful minors and random separators} \label{sec:careful}

For graphs $H$ and $G$, one says that {\em $H$ is a minor of $G$} if there are pairwise-disjoint connected subsets $\{ A_u \subseteq V_G : u \in V_H \}$ such that $\{u,v\} \in E_H \implies E_G(A_u, A_v) \neq \emptyset$. We will sometimes refer to the sets $\{A_u\}$ as {\em supernodes.} Say that {\em $H$ is a strict minor of $G$} if the stronger condition $\{u,v\} \in E_H \iff E_G(A_u, A_v) \neq \emptyset$ holds. Finally, we say that {\em $H$ is a careful minor of $G$} if $\dot{H}$ is a strict minor of $G$.

\medskip

The next result explains the significance of careful minors for region intersection graphs. We prove it in the next section.

\begin{lemma}\label{lem:careful-minor} If $\hat G \in \rig(G)$ and $\hat G$ has a careful $H$-minor, then $G$ has an $H$-minor. \end{lemma}

We now state the main result of this section; its proof occupies Sections \ref{sec:chopping}--\ref{sec:warm}.

\begin{theorem}\label{thm:random-sep} For any $h \geq 1$, the following holds. Suppose that $G$ excludes a careful $K_h$ minor. Then there is a number $\alpha \leq O(h^2)$ such that for any $\omega : V_G \to \R_+$ and $\Delta > 0$, the conformal graph $(G,\omega)$ admits an $(\alpha,\Delta)$-random separator. \end{theorem}

Applying \pref{lem:careful-minor} immediately yields the following.

\begin{corollary}\label{cor:random-sep} Suppose that $G$ excludes a $K_h$ minor and $\hat G \in \rig(G)$. Then there is a number $\alpha \leq O(h^2)$ such that for any $\omega : V_{\hat G} \to \R_+$ and $\Delta > 0$, the conformal graph $(\hat G,\omega)$ admits an $(\alpha,\Delta)$-random separator. \end{corollary}

The proof of \pref{thm:random-sep} is based on a procedure that iteratively removes random sets of vertices from the graph in rounds.  It is modeled after the argument of \cite{FT03} which is itself based on \cite{KPR93}. For an exposition of the latter argument, one can consult the book \cite[Ch. 3]{Ostro13}.

\subsection{Careful minors in rigs} \label{sec:careful-minors}

The next lemma clarifies slightly the structure of careful minors.

\begin{lemma}\label{lem:careful-struct} $G$ has a careful $H$-minor if and only if there exist pairwise-disjoint connected subsets $\{B_u \subseteq V_G : u \in V_H\}$ and distinct vertices $W = \left\{ w_{xy} \in V_G \setminus \bigcup_{u \in V_H} B_u : \{x,y\} \in E_H\right\}$ such that \begin{enumerate} \item $E_G(B_u, B_v) = \emptyset$ for $u,v \in V_H$ with $u \neq v$. \item $W$ is an independent set. \item For every $\{x,y\} \in E_H$, it holds that \begin{equation}\label{eq:careful-struct-prop} E_G(w_{xy}, B_u) \neq \emptyset \iff u \in \{x,y\}\,. \end{equation} \end{enumerate} \end{lemma}

\begin{proof} The ``only if'' direction is straightforward.  We now argue the other direction.

Let $\{A_u \subseteq V_G : u \in V_{\dot{H}} \}$ witness a strict $\dot{H}$-minor in $G$. For every $\{x,y\} \in E_H$, there exists a simple path $\gamma_{xy}$ with one endpoint in $A_x$, one endpoint in $A_y$, and whose internal vertices satisfy $\gamma^{\circ}_{xy} \subseteq A_{m_{xy}}$ and $\gamma_{xy}^{\circ} \neq \emptyset$, where $m_{xy} \in V_{\dot{H}}$ is the vertex subdividing the edge $\{x,y\}$.

Choose some vertex $w_{xy} \in \gamma^{\circ}_{xy}$.  Removal of $w_{xy}$ breaks the graph $G[A_x \cup A_y \cup \gamma_{xy}]$ into two connected components; define these as $B_x$ and $B_y$ (so that $A_x \subseteq B_x$ and $A_y \subseteq B_y$).
Property (1) is verified by strictness of the $\dot{H}$ minor and the fact that the the non-subdivision vertices $V_{\dot{H}} \setminus \left\{m_{xy} : \{x,y\} \in V_H \right\}$ form an independent set in $\dot{H}$. Similarly, properties (2) and (3) follow from strictness of the $\dot{H}$ minor and the fact that $N_{\dot{H}}(m_{xy}) = \{x,y\}$ for $\{x,y\} \in E_H$. \end{proof}

We now prove that if $\hat G \in \rig(G)$, then careful minors in $\hat G$ yield minors in $G$.

\begin{proof}[Proof of \pref{lem:careful-minor}] Let $\{ R_v \subseteq V_G : v\in V_{\hat G}\}$ be a set of regions realizing $\hat G$. Assume that $\hat G$ has a careful $H$-minor and let $\{B_u \subseteq V_{\hat G} : u \in V_H\}$ and $W=\{w_{xy} : \{x,y\} \in E_H\}$ be the sets guaranteed by \pref{lem:careful-struct}.

For $u \in V_H$, define \[ A_u = \bigcup_{v \in B_u} R_v\,. \] Since $B_u$ is connected in $\hat G$ and the regions $\{R_v : v \in V_{\hat G}\}$ are each connected in $G$, it follows that $A_u$ is connected in $G$.

Let us verify that the sets $\{A_u : u \in V_H\}$ are pairwise disjoint. If $x \in A_u \cap A_v$ for $u \neq v$, then there must be regions $R_a$ and $R_b$ with $a \in B_u, b \in B_v$ and $x \in R_a \cap R_b$.  This would imply $\{a,b\} \in E_{\hat G}$, but \pref{lem:careful-struct}(1) asserts that $E_{\hat G}(B_u,B_v)=\emptyset$.

We will show that there exist pairwise vertex-disjoint paths $\left\{\gamma_{uv} \subseteq V_G : \{u,v\} \in E_H\right\}$ with \begin{equation}\label{eq:disjoint-gamma} \gamma_{uv} \cap \left(\bigcup_{x\in V_H} A_x\right) \subseteq A_u \cup A_v\,, \end{equation} and such that $\gamma_{uv}$ connects $A_u$ to $A_v$.  This will yield the desired $H$-minor in $G$.

Fix $\{u,v\} \in E_H$. From \pref{lem:careful-struct}(3), we know that the connected set $R_{w_{uv}}$ shares a vertex with $A_u$ and also
shares a (different) vertex with $A_v$. Thus we can choose $\gamma_{uv}$ as above with $\gamma_{uv} \subseteq R_{w_{uv}}$. Note that \pref{lem:careful-struct}(2) (in particular, \eqref{eq:careful-struct-prop}) also yields $R_{w_{uv}} \cap A_x = \emptyset $ for any $x \in V_H \setminus \{u,v\}$, verifying \eqref{eq:disjoint-gamma}.

Thus we are left to verify that the sets $\{R_{w_{uv}} : \{u,v\} \in E_H\}$ are pairwise vertex-disjoint. But this also follows from \eqref{eq:careful-struct-prop}, specifically the fact that $W=\{w_{uv} : \{u,v\} \in E_H\}$ is an independent set in $\hat G$. \end{proof}

\subsection{Chopping trees} \label{sec:chopping}

Observe that by a trivial approximation argument,
it suffices to prove \pref{thm:random-sep} for any conformal metric $\omega : V_G \to (0,\infty)$, i.e.,
one that satisfies
\begin{equation}\label{eq:nonzero}
   \omega(v) > 0\quad\forall v \in V_G
\end{equation}

Let us now fix such a conformal metric $\omega : V_G \to (0,\infty)$ on $G$ and a number $\Delta > 0$. Fix an arbitrary ordering $v_1, v_2, \ldots, v_{|V_G|}$ of $V_G$ in order to break ties in the argument that follows.

\begin{figure} \begin{center}
\subfigure[Illustration of a fat sphere]{ \includegraphics[width=6cm]{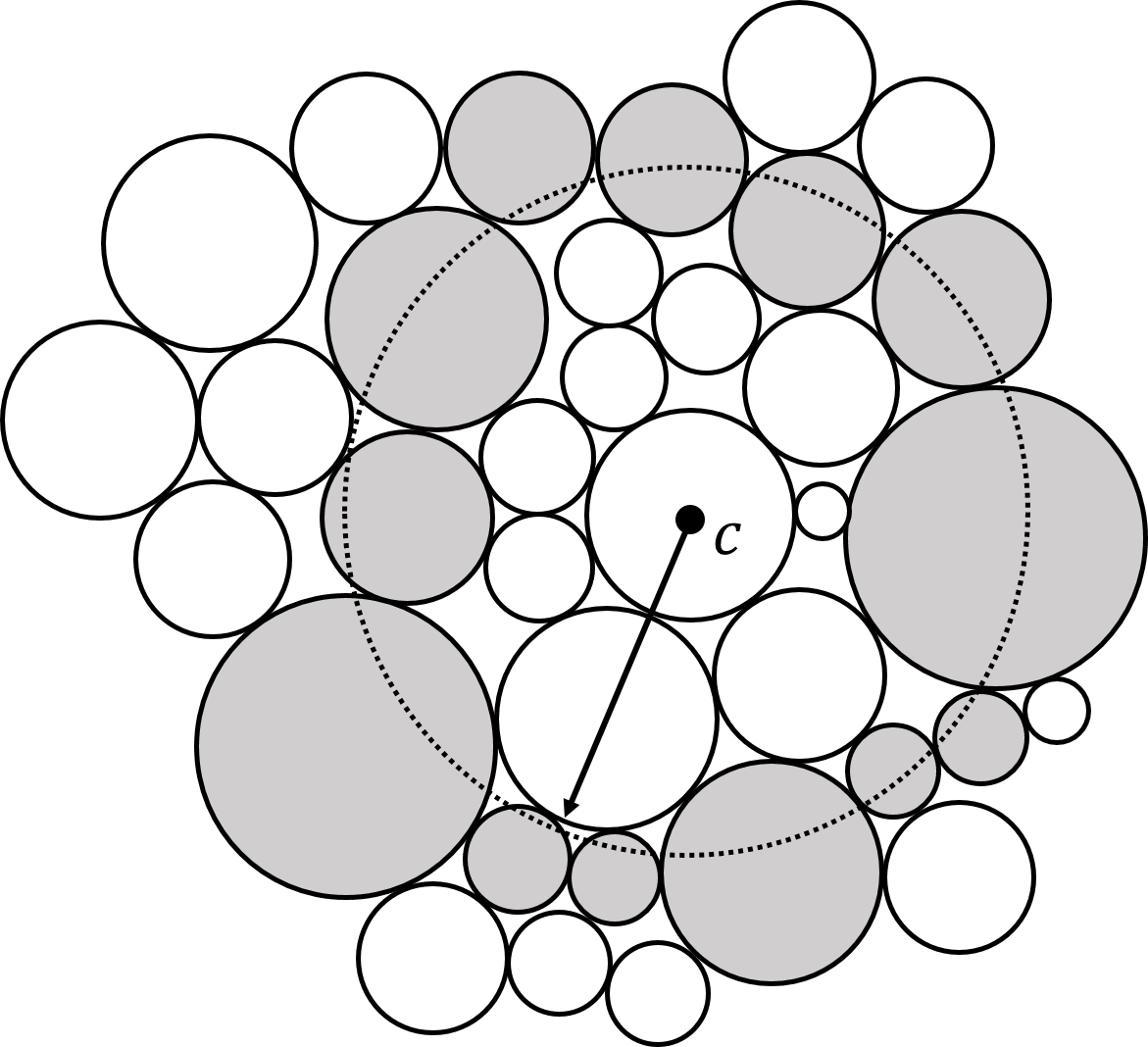}\label{fig:fatsphere} } \hspace{0.6in}
\subfigure[Chopping a graph into subgraphs]{ \makebox[5cm]{ \includegraphics[width=5.5cm]{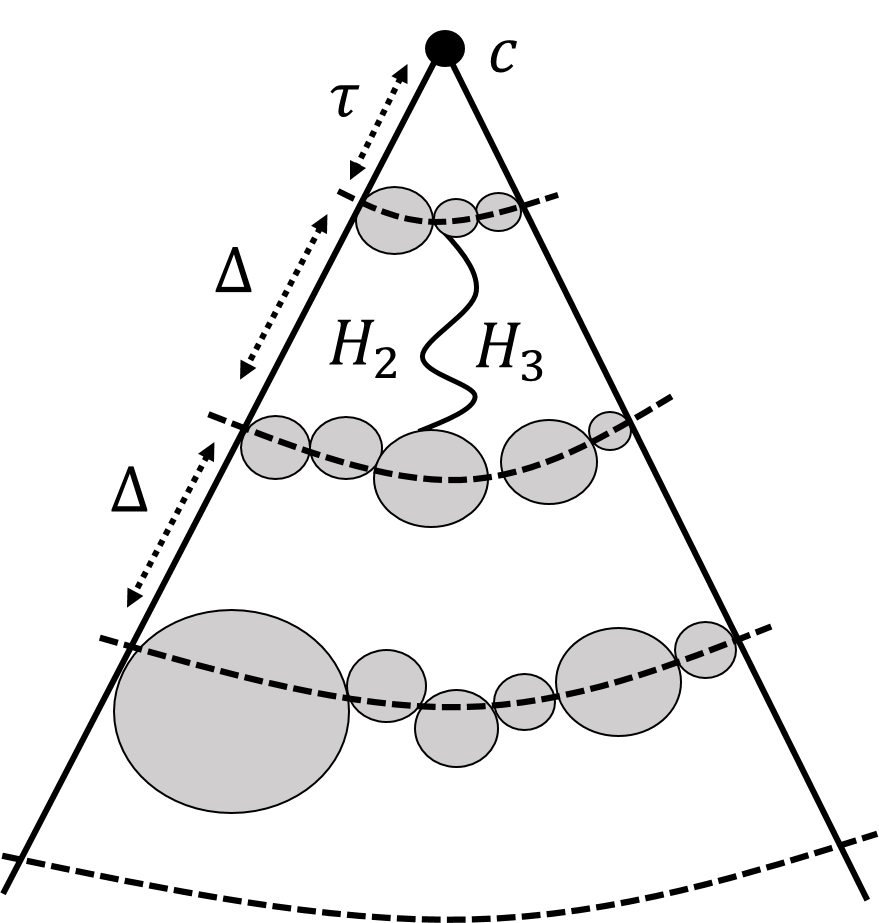}\label{fig:chop}} } \caption{The chopping procedure} \end{center} \end{figure}

We will use $\Indc(G)$ to denote the collection of all connected, induced subgraphs of $G$. For such a subgraph $H \in \Indc(G)$, we use $\dist_{\omega}^H$ to denote the induced distance coming from the conformal metric $(H,\omega|_{V_H})$. For $c \in V_H$ and $R > 0$, let us define the skinny ball, fat ball, and fat sphere, respectively: \begin{align*} {\cB}^H_{\omega}(c,R) &= \left\{ v \in V_H : \dist_{\omega}^H(c,v) < R - \frac12 \omega(v) \right\},\\ \vvB_{\omega}^H(c,R) &= \left\{ v \in V_H : \dist_{\omega}^H(c,v) \leq R + \frac12 \omega(v) \right\}, \\
\vvS_{\omega}^H(c,R) &= \left\{ v \in V_H : R \in \left[\dist_{\omega}^H(c,v) - \tfrac12 \omega(v), \dist_{\omega}^H(c,v) + \tfrac12 \omega(v)\right]\right\} \\ &= \vvB_{\omega}^H(c,R) \setminus \cB_{\omega}^H(c,R)\,. \end{align*} See \pref{fig:fatsphere} for a useful (but non-mathematical) illustration where one imagines a vertex $v \in V_H$ as a disk of radius $\frac12 \omega(v)$. Note that ${\cB}_{\omega}^H(c,R)$ is the connected component of $c$ in the graph $H[V_H \setminus \vvS_{\omega}^H(c,R)]$.
The next fact requires our assumption \eqref{eq:nonzero}.

\begin{fact}\label{fact:unique} If $\gamma \subseteq V_H$ is a $\dist^H_{\omega}$-shortest path emanating from $c \in V_H$, then for every $R > 0$, it holds that $|\gamma \cap \vvS_{\omega}^H(c,R)| \leq 1$. \end{fact}

For $c \in V_H$ and $\tau \in [0,\Delta]$, let $\cut_{\Delta}(H,c;\tau) = \bigcup_{k \in \Z_+} \vvS_{\omega}^H(c, \tau+k\Delta)$. We define  $\chop_{\Delta}(H,c;\tau)$ as the collection of connected components of the graph $H\left[V_H \setminus\cut_{\Delta}(H,c;\tau)\right]$. See \pref{fig:chop}.

The next lemma is straightforward.

\begin{lemma}\label{lem:random-chop} If $\tau \in [0,\Delta]$ is chosen uniformly at random, then for every $v \in V_H$ and $R \geq 0$, \[ \Pr[\cB_{\omega}^H(v,R) \cap \cut_{\Delta}(H,c;\tau) = \emptyset] \geq 1-\frac{2R}{\Delta}\,. \] \end{lemma}

A {\em $\Delta$-chopping tree of $(G,\omega)$} is a rooted, graph-theoretic tree $\cT(\sigma)$ for some $\sigma : \Indc(G) \to [0,\Delta]$. The nodes of $\cT(\sigma)$ are triples $(H,c,j)$ where $H \in \Indc(G)$, $c \in V_H$, and $j \in \Z_+$. We refer to $c$ as the {\em center} of the node and $j$ as its {\em depth.} We now define $\cT(\sigma)$ inductively (by depth) as follows.

The root of $\cT=\cT(\sigma)$ is $(G,v_1,0)$. For a node $\lambda = (H,c,j)$ of $\cT$, we let $\vec{c}_{\cT}(\lambda)$ denote the sequence of centers encountered on the path from $\lambda$ to the root of $\cT$, not including $\lambda$ itself. If $\chop_{\Delta}(H,c;\sigma(H)) = \emptyset$, then $\lambda$ has no children.

Otherwise, if $\chop_{\Delta}(H,c;\sigma(H))=\{H_i : i \in I\}$, the children of $(H,c,j)$ are $\{(H_i, c_i,j+1)\}$, where \begin{equation}\label{eq:argmax} c_i = \argmax_{x \in V_{H_i}} \dist_{\omega}^G(x, \vec{c}_{\cT}(\lambda) \cup \{c\})\,. \end{equation} In other words, $c_i$ is chosen as the point of $V_{H_i}$ that is furthest from the centers of its ancestors in the {\em ambient metric $\dist_{\omega}^G$.} For concreteness, if the maximum in \eqref{eq:argmax} is not unique, we choose the first vertex (according to the ordering of $V_G$) that achieves the maximum.

A final definition:  We say that a node $\lambda=(H,c,j)$ of a chopping tree $\cT$ is {\em $\beta$-spaced} if the value of the maximum in \eqref{eq:argmax} is at least $\beta$, i.e.,
 \[\dist_{\omega}^G(c,\vec{c}_{\cT}(\lambda)) \geq \beta\,.\]

 Note that the nodes in each level of $\cT(\sigma)$ correspond to the connected components
 that result after removing a subset of nodes from $G$.
 We state the following consequence.

\begin{lemma}\label{lem:alevel}
   Suppose that $\cT(\sigma)$ is a $\Delta$-chopping tree for some $\Delta > 0$.
   Consider an integer $k \geq 0$,
   and let $\{H_i : i \in I\} \subseteq \Indc(G)$ denote the collection of
   induced subgraphs occuring in the depth-$k$ nodes of $\cT(\sigma)$.
   Then each $H_i$ is a unique connected component in the induced graph $G\left[\bigcup_{i \in I} V_{H_i}\right]$.
\end{lemma}

We now state the main technical lemma on chopping trees.  The proof appears in \pref{sec:warm}.

\begin{lemma}\label{lem:spaced} Consider any $h \geq 1$ and $\Delta > 0$. Assume the following conditions hold: \begin{enumerate} \item $\max_{v \in V_G} \omega(v) \leq \Delta$ \item $\cT$ is a $\Delta$-chopping tree of $(G,\omega)$. \item There exists a $21h\Delta$-spaced node of $\cT$ at depth $h-1$. \end{enumerate} Then $G$ contains a careful $K_h$ minor. \end{lemma}

Finally, we have the following analysis of a random chopping tree.

\begin{lemma}\label{lem:random-chopping} For any $k \geq 1$, the following holds. Suppose that $\sigma : \Indc(G) \to [0,\Delta]$ is chosen uniformly at random. Let $\{H_i : i \in I\} \subseteq \Indc(G)$ denote the collection of induced subgraphs occurring in the depth-$k$ nodes of $\cT(\sigma)$. For any $v \in V_G$, \[ \Pr\left[\cB_{\omega}^G(v,R) \subseteq \bigcup_{i \in I} V_{H_i}\right] \geq 1 - 2k \frac{R}{\Delta}\,. \] \end{lemma}

\begin{proof} Note that since $\cB_{\omega}^G(v,R)$ is a connected set and
   there are no edges between $H_i$ and $H_j$ for $i \neq j$ (cf. \pref{lem:alevel}),
we have \[ \cB_{\omega}^G(v,R) \subseteq \bigcup_{i \in I} V_{H_i} \iff \exists i \in I \textrm{ s.t. } \cB_{\omega}^G(v,R) \subseteq V_{H_i}\,. \] The set $\cB_{\omega}^G(v,R)$ experiences at most $k$ random chops, and the probability it gets removed in any one of them is bounded in \pref{lem:random-chop}.  The desired result follows by observing that if $H \in \Indc(G)$ satisfies $\cB_{\omega}^G(v,R) \subseteq V_H$, then $\cB_{\omega}^H(v,R)=\cB_{\omega}^G(v,R)$. \end{proof}

\subsection{The random separator construction}

We require an additional tool before proving \pref{thm:random-sep}. For nodes that are not well-spaced, we need to apply one further operation.

If $H \in \Indc(G)$ and $\vec c = (c_1, c_2, \ldots, c_k) \in V_G^k$ and $\vec{\tau} = (\tau_1, \tau_2, \ldots, \tau_k) \in \R_+^k$, we define a subset $\shatter_{\Delta}(H,\vec{c},\vec{\tau}) \subseteq \Indc(G)$ as follows. Define \[ \shards_{\Delta}(H,\vec{c},\vec{\tau}) = V_H \cap \bigcup_{i=1}^k \vvS_{\omega}^G(c_i, \Delta+\tau_i)\,, \] and let $\shatter_{\Delta}(H,\vec{c},\vec{\tau})$ be the collection of connected components of $H[V_H \setminus \shards(H,\vec{c},\vec{\tau})]$. The next two lemmas are straightforward consequences of this construction.

\begin{lemma}\label{lem:diambound} If every $v \in V_H$ satisfies $\min \{ \dist_{\omega}^G(v,c_i) : i=1,\ldots,k\} \leq \Delta$,
then for every $H' \in \shatter_{\Delta}(H,\vec{c},\vec{\tau})$, it holds that \[\diam_{\omega}^G(V_{H'}) \leq 2(\Delta + \max \vec{\tau})\,.\] \end{lemma}

\begin{lemma}\label{lem:shard-prob} For any $\vec{c} \in V_G^k$ and any $\Delta' > 0$, if $\vec{\tau} \in [0,\Delta']^k$ is chosen uniformly at random, then for any $v \in V_H$ and $R \geq 0$, \[ \Pr[\cB_{\omega}^H(v,R) \cap \shards_{\Delta}(H,\vec{c},\vec{\tau}) = \emptyset] \geq 1 - 2 k \frac{R}{\Delta'}\,. \] \end{lemma}

\begin{proof}[Proof of \pref{thm:random-sep}] We may assume that if $v \in V_G$ has $\omega(v) > \Delta$, then $v \in \bm{S}$. Indeed, denote
\[
      Q = \{ v \in V_G : \omega(v) > \Delta\}\,.
\]
      If we can produce an $(\alpha,\Delta)$-separator for each of the connected components of $G[V_G \setminus Q]$, then taking the union of those separators together with $Q$ yields a $(2\alpha,\Delta)$-separator of $G$. We may therefore assume that $\max_{v \in V_G} \omega(v) \leq \Delta$.

Assume now that $G$ excludes a careful $K_h$ minor. Let $\cT=\cT(\sigma)$ be the $\Delta$-chopping tree of $(G,\omega)$ with $\sigma : \Indc(G) \to [0,\Delta]$ chosen uniformly at random.

Let $D_{h-1} = \{\lambda_i = (H_i, c_i, h) : i \in I\}$ be the collection of depth-$(h-1)$ nodes of $\cT(\sigma)$. Let $\bm{S}_1 = V_G \setminus \bigcup_{i \in I} V_{H_i}$.  By construction, the graphs $\{H_i\}$ are precisely the connected components of $G[V\setminus \bm{S}_1]$ (and they occur without repetition, i.e., $H_i \neq H_j$ for $i \neq j$).

By \pref{lem:random-chopping}, for any $v \in V_G$ and $R \geq 0$, we have \begin{equation}\label{eq:prob1} \Pr[\cB_{\omega}^G(v, R) \cap \bm{S}_1=\emptyset ] \geq 1 - 2h \frac{R}{\Delta}\,. \end{equation} Define \begin{equation*} \bm{S}_2 = \bigcup_{i \in I} \shards_{21h\Delta}(H_i, \vec{c}_{\cT}(\lambda_i); \vec{\tau})\,, \end{equation*} where $\vec{\tau} \in [0,\Delta]^h$ is chosen uniformly at random.  From \pref{lem:shard-prob}, for any $i \in I$, $v \in V_{H_i}$, and $R \geq 0$, we have
 \begin{equation}\label{eq:prob2}
\Pr[\cB_{\omega}^{H_i}(v, R) \cap \bm{S}_2=\emptyset] \geq 1- 2h\frac{R}{\Delta}\,. \end{equation}

So consider $v \in V_G$ and $R \geq 0$.  If $\cB_{\omega}^G(v,R) \cap \bm{S}_1 = \emptyset$, then $\cB_{\omega}^G(v,R) \subseteq V_{H_i}$ for some $i \in I$, and in that case $\cB_{\omega}^{G}(v,R)=\cB_{\omega}^{H_i}(v,R)$.  Therefore \eqref{eq:prob1} and \eqref{eq:prob2} together yield \begin{equation}\label{eq:prob3} \Pr[\cB_{\omega}^G(v,R) \cap (\bm{S}_1 \cup \bm{S}_2) = \emptyset] \geq 1 - 4h \frac{R}{\Delta}\,. \end{equation}

Moreover, the collection of induced subgraphs \[ \cH = \bigcup_{i \in I} \shatter_{21h\Delta}(H_i, \vec{c}_{\cT}(\lambda_i); \vec{\tau}) \] is precisely the set of connected components of $G[V \setminus (\bm{S}_1 \cup \bm{S}_2)]$.

We are thus left to bound $\diam_{\omega}^G(V_H)$ for every $H \in \cH$. Consider a node $\lambda_i \in D_{h-1}$. Since $G$ excludes a careful $K_h$ minor, \pref{lem:spaced} implies that $\lambda_i$ is not $21h\Delta$-spaced. It follows that \[ \max_{v \in V_{H_i}} \dist_{\omega}^G(v, \vec{c}_{\cT}(\lambda_i)) \leq \dist_{\omega}^G(c_i, \vec{c}_{\cT}(\lambda_i)) \leq 21 h \Delta\,, \] based on how $c_i$ is chosen in \eqref{eq:argmax}.  Therefore \pref{lem:diambound} implies that for every $H \in \shatter_{21 h\Delta}(H_i, \vec{c}_{\cT}(\lambda_i);\vec{\tau})$, we have $\diam_{\omega}^G(V_{H}) \leq 2(21 h+1) \Delta$.

We conclude that every connected component $H$ of $G[V \setminus (\bm{S}_1\cup \bm{S}_2)]$ has $\diam_{\omega}^G(V_H) \leq (42h+2)\Delta$. Combining this with \eqref{eq:prob3} shows that $\bm{S}_1 \cup \bm{S}_2$ is a $\left(4h(42h+2),(42h+2)\Delta\right)$-random separator, yielding the desired conclusion.
(Note that establishing the existence of a $(c\alpha, c\Delta)$-random separator for every $\Delta > 0$ implies
the existence of an $(\alpha,\Delta)$-random separator for every $\Delta > 0$ by homogeneity.)
\end{proof}

\subsection{A diameter bound for well-spaced subgraphs} \label{sec:warm}

Our goal now is to prove \pref{lem:spaced}.

\begin{lemma}[Restatement of \pref{lem:spaced}]\label{lem:spaced-copy} Consider any $h \geq 1$ and $\Delta > 0$. Assume the following conditions hold: \begin{enumerate} \item[(A1)] $\max_{v \in V_G} \omega(v) \leq \Delta$ \item[(A2)] $\cT$ is a $\Delta$-chopping tree of $(G,\omega)$. \item[(A3)] There is a $21 h\Delta$-spaced node of $\cT$ at depth $h-1$. \end{enumerate} Then $G$ contains a careful $K_h$ minor. \end{lemma}

In order to enforce the properties of a careful minor, we will need a way to ensure that there are no edges between certain vertices.  The following simple fact will be the primary mechanism.

\begin{lemma}\label{lem:noedge} Suppose $H \in \Indc(G)$ and $\max_{v \in V_H} \omega(v) \leq \Delta$. If $u,v \in V_H$ satisfy $\dist_{\omega}^H(u,v) > \Delta$, then $\{u,v\} \notin E_G$. \end{lemma}

\begin{proof} Since $H$ is an induced subgraph, $\{u,v\} \in E_G \implies \{u,v\} \in E_H$. And then clearly $\dist_{\omega}^H(u,v) = \frac{\omega(u)+\omega(v)}{2} \leq \Delta\,.$ \end{proof}

\begin{proof}[Proof of \pref{lem:spaced-copy}] We will construct a careful $K_h$ minor inductively. Recall that a careful $K_h$ minor is a strict minor of the subdivision $\dot{K}_h$. We use the notation $\{A_u\}$ for the supernodes corresponding to original vertices of $K_h$. The supernodes corresponding to subdivision vertices will be single nodes of $V_G$ which we denote $\{ w_{uv} : \{u,v\} \in E_{K_h}\}$.

Let $\lambda=(H,c,h-1)$ be a $21h\Delta$-spaced node in $\cT$, and denote by $\lambda=\lambda_1, \lambda_2, \ldots, \lambda_{h}$ the
sequence of nodes of $\cT$ on the path from $\lambda$ to the root of $\cT$.
For each $t=1,\ldots,h$, write $\lambda_t=(H_t,c_t,h-t)$.

Observe that since $\lambda$ is $21 h \Delta$-spaced, it also holds that $\lambda_t$ is $21 h\Delta$-spaced for each $t=1,\ldots,h$.  (The property only becomes stronger for children in $\cT$.)  Hence, \begin{equation}\label{eq:center-sep} \dist_{\omega}^G(c_i, c_j) \geq 21h\Delta \textrm{ for all $i,j \in \{1,\ldots,h\}$ with $i\neq j$}. \end{equation}
We will show, by induction on $t$, that $H_t$ contains a careful $K_{t}$ minor for $t=1,\ldots,h$.

For the sake of the induction, we will need to maintain some additional properties that we now describe. The first three properties simply ensure that we have found a strict $\dot{K}_{t}$ minor. Let us use the numbers $\{1,2,\ldots,t\}$ to index the vertices of $V_{K_t}$. We will show there exist sets $\{A^t_u \subseteq V_{H_t} : u \in V_{K_t}\}$ and $W^t = \left\{ w_{uv} \in V_{H_t} : \{u,v\} \in E_{K_t}\right\} \subseteq V_{H_t}$ with the following properties: \begin{enumerate} \item[P1.] The sets $\{A^t_u : u \in V_{K_t}\}$ are connected and mutually disjoint, and $E_G(A^t_u,A^t_v)=\emptyset$ for $u \neq v$. \item[P2.] The set $W^t$ is an independent set in $G$. \item[P3.] For all $\{u,v\} \in E_{K_t}$ it holds that $E_G(w_{uv}, A^t_x) \neq \emptyset \iff x \in \{u,v\}$. \item[P4.] For every $u \in V_{K_t}$, there is a representative $r^t_u \in A^t_u$ such that $\dist_{\omega}^G(r^t_u, c_u) \leq 10 t \Delta$. \item[P5.] For all $u \neq v \in V_{K_t}$, we have $\dist_{\omega}^G(r^t_u,r^t_v) \geq 21(h-t) \Delta\,.$ \item[P6.] For every $u \in V_{K_t}$, it holds that \[ \dist_{\omega}^{H_t}\left(r^t_u, W^t \cup \bigcup_{v \neq u} A^t_v\right) \geq 3\Delta\,. \] \end{enumerate}

In the base case $t=1$, take $A_1^1 = \{c_1\}$ and $r^1_1=c_1$, and $W^1=\emptyset$. It is easily checked that these choices satisfy (P1)--(P6). So now suppose that for some $t \in \{1,2,\ldots,h-1\}$, we have objects satisfying (P1)--(P6). We will establish the existence of objects satisfying (P1)--(P6) for $t+1$.

\begin{figure} \begin{center} \makebox[18.5cm]{ \includegraphics[width=16cm]{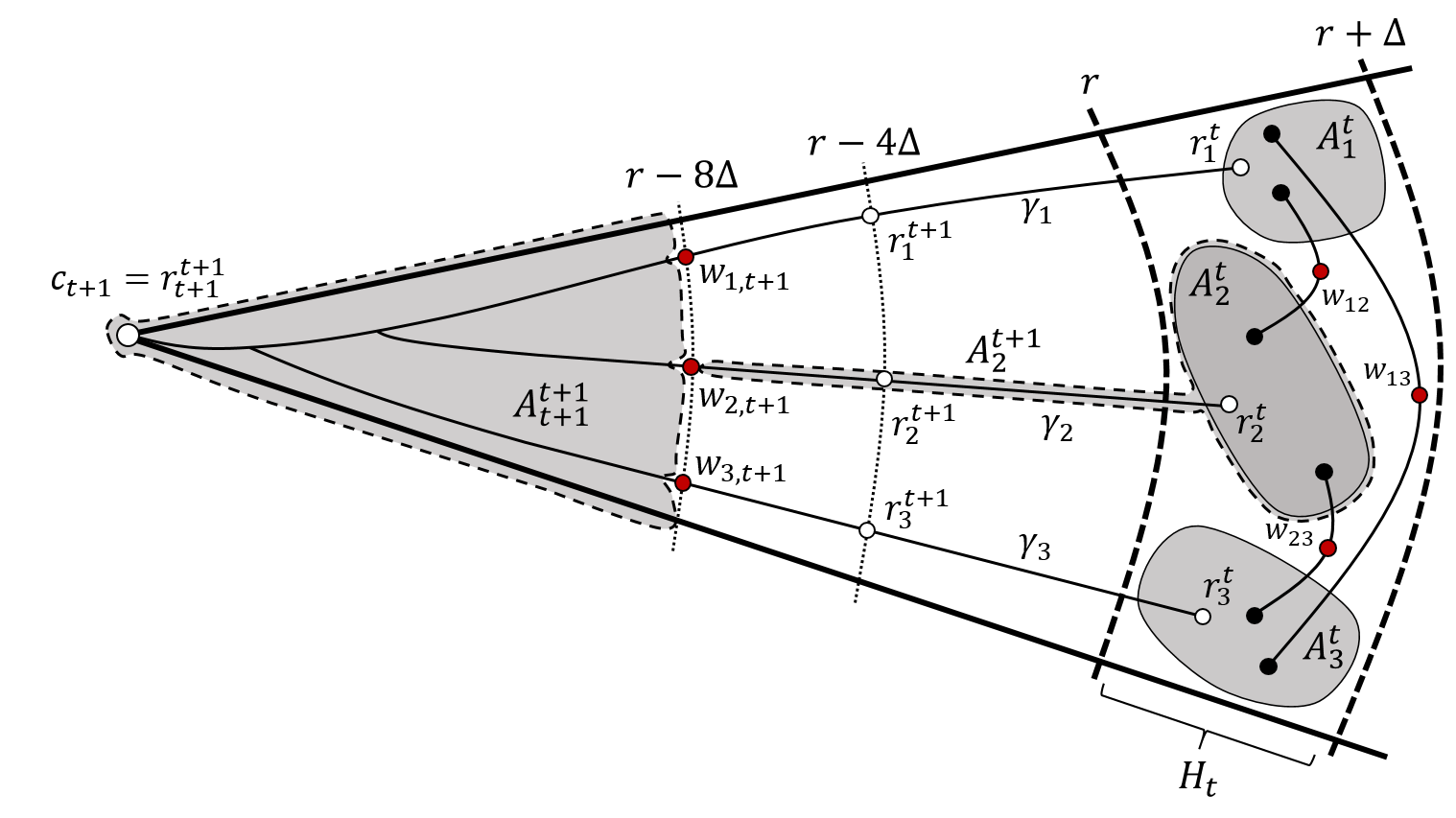}\hspace{3cm}} \vspace{-1cm} \caption{Construction of a careful $K_h$ minor.\label{fig:minor}} \end{center} \end{figure}

It may help to consult \pref{fig:minor} for the inductive step. Recall that, by construction, $H_t \in \chop_{\Delta}(H_{t+1}, c_{t+1}; r)$ for some $r > 0$.
Note that since $\lambda_t$ is $21 h\Delta$-spaced, it holds that
$\dist_{\omega}^G(c_t,c_{t+1}) \geq 21 h \Delta$ (cf. \eqref{eq:center-sep}), and therefore
(recalling that $\max_{v \in V_G} \omega(v) \leq \Delta$)
\begin{equation}\label{eq:rbound} r \geq 20h\Delta\,. \end{equation}

For each $i=1,2,\ldots,t$, let $\gamma_i$ denote a $\dist_{\omega}^{H_{t+1}}$-shortest-path from $r^t_i$ to $c_{t+1}$. Let $r_i^{t+1}$ denote the unique element in $\gamma_i \cap \vvS^{H_{t+1}}_{\omega}(c_{t+1}, r-4\Delta)$, and let $w_{i,t+1}$  denote the unique element in $\gamma_i \cap \vvS^{H_{t+1}}_{\omega}(c_{t+1},r-8\Delta)$. (Recall \pref{fact:unique}.) Define: \begin{align*} A_i^{t+1} &= A_i^t \cup \gamma_i \setminus \vvB_{\omega}^{H_{t+1}}(c_{t+1},r-8\Delta),\quad 1 \leq i \leq t \\ A_{t+1}^{t+1} &= {\cB}_{\omega}^{H_{t+1}}(c_{t+1}, r-8\Delta)\,. \end{align*}

\begin{lemma}\label{lem:far-prelim} For every $i,j \in \{1,2,\ldots, t\}$ with $i\neq j$, the following holds:  $\gamma_i \cap A_j^t = \emptyset$, $\gamma_i \cap W^t = \emptyset$, and $E_G(\gamma_i, A_j^t) = E_G(\gamma_i, W^t) = \emptyset$. \end{lemma}

\begin{proof} Observe that $\len_{\omega}(\gamma_i \cap V_{H_t}) \leq \Delta$ and $\gamma_i$ emanates from $r_i^t$. (P6) implies that \[\dist_{\omega}^{H_t}(r_i^t, A_j^{t} \cup W^t) \geq 3\Delta\,,\] and thus $\gamma_i \cap (A_j^{t} \cup W^t) = \emptyset$ for $i \neq j$.

Note furthermore that $\vvS_{\omega}^{H_{t+1}}(c_{t+1}, r)$ separates $\gamma_i \setminus V_{H_t}$ from $V_{H_t}$ in $H_{t+1}$. Thus we need only prove that $E_G(\gamma_i \setminus {\cB}_{\omega}^{H_{t+1}}(c_{t+1},r), A_j^t \cup W^t) = \emptyset$. But we have \begin{align*} \dist_{\omega}^{H_{t+1}}(\gamma_i \setminus {\cB}_{\omega}^{H_{t+1}}(c_{t+1},r), A_j^t \cup W^t) & \geq \dist_{\omega}^{H_{t}}(r_i^t, A_j^t \cup W^t) - \len_{\omega}(\gamma_i \cap V_{H_t}) - \tfrac12 \max_{v \in \vvS_{\omega}^{H_{t+1}}(c_{t+1}, r)} \omega(v) \\& \stackrel{\textrm{(A1)}}{\geq}
 \dist_{\omega}^{H_{t}}(r_i^t, A_j^t \cup W^t) -3\Delta/2 \\ &\stackrel{\textrm{(P6)}}{>} \Delta\,.
\end{align*} Therefore \pref{lem:noedge} yields $E_G(\gamma_i \setminus {\cB}_{\omega}^{H_{t+1}}(c_{t+1},r), A_j^t \cup W^t) = \emptyset$, and this suffices to prove $E_G(\gamma_i, A_j^t \cup W^t)=\emptyset$ for $i \neq j$. \end{proof}

Let us verify the six properties (P1)--(P6) above in order.

 \begin{enumerate}
\item Consider first the sets $A_i^{t+1}$ for $i=1,2,\ldots,t$. For $i \neq j$, we have $\gamma_i \cap A_j^t = \emptyset$ and $E_G(\gamma_i, A_j^t)=\emptyset$ from \pref{lem:far-prelim}. Next consider the sets $\gamma_i \cap A_i^{t+1}$ and $\gamma_j \cap A_j^{t+1}$ for $i \neq j$ and $i,j \leq t$. By (P5), we have $\dist_{\omega}^G(r_i^t, r_j^t) \geq 14(h-t)\Delta$, and thus $\dist_{\omega}^G(\gamma_i \cap A_{i}^{t+1}, \gamma_j \cap A_j^{t+1}) \geq 21(h-t)\Delta -2\cdot 9 \Delta > \Delta$. Hence $\gamma_i \cap A_i^{t+1}$ and $\gamma_j \cap A_j^{t+1}$ are disjoint and by \pref{lem:noedge}, $E_G(\gamma_i \cap A_i^{t+1}, \gamma_j \cap A_j^{t+1}) = \emptyset$.

Finally, observe that $A_{t+1}^{t+1} \cap A_i^{t+1} = \emptyset$ and $E_G(A_{t+1}^{t+1}, A_i^{t+1})=\emptyset$ for $i \leq t$ because $\vvS_{\omega}^{H_{t+1}}(c_{t+1},r-8\Delta)$ separates $A_{t+1}^{t+1}$ from $\bigcup_{i \leq t} A_i^{t+1}$ in $H_{t+1}$.

\item Observe that $\vvS_{\omega}^{H_{t+1}}(c_{t+1},r-4\Delta)$ and $\vvS_{\omega}^{H_{t+1}}(c_{t+1},r-8\Delta)$ are disjoint because $\max_{v \in V_{H_{t+1}}} \omega(v) \leq \Delta$ by assumption (A1).  It follows that $\vvS_{\omega}^{H_{t+1}}(c_{t+1},r-4\Delta)$ separates $W^{t+1} \setminus W^t$ from $W^t$ in $H_{t+1}$.  We thus need only verify that $W^{t+1}\setminus W^t$ is an independent set.

To this end, observe that \begin{equation}\label{eq:wrep-dist} \dist_{\omega}^G(w_{i,t+1}, r_i^t) \leq 9\Delta+\Delta/2 \leq 9.5\Delta\,, \end{equation} where we have again employed (A1). This implies that for $i \neq j$, \[\dist_{\omega}^G(w_{i,t+1}, w_{j,t+1}) \geq \dist_{\omega}^G(r_i^t, r_j^t) - 2\cdot 9.5\Delta\stackrel{\textrm{(P5)}}{\geq} 21(h-t)\Delta - 19 \Delta > \Delta\,,\] hence \pref{lem:noedge} implies that $W_{t+1} \setminus W_t$ is indeed an independent set.

\item The facts that $E_G(A_{t+1}^{t+1}, w_{i,t+1}) \neq \emptyset$ and $E_G(A_i^{t+1}, w_{i,t+1}) \neq \emptyset$ for $i \leq t$ both follow immediately from the construction ($w_{i,t+1}$ is a separator vertex on the path $\gamma_i$ connecting $r_i^t$ to $c_{t+1}$).

We are left to verify that for every $\{u,v\} \in V_{K_{t+1}}$ and $x \notin \{u,v\}$, we have $E_G(w_{uv}, A_x^{t+1})=\emptyset$. We argue this using three cases: \begin{itemize} \item For $i \leq t$,  $E_G(w_{i,t+1}, A_x^t) = \emptyset$:

  This follows because $E_G(w_{i,t+1}, V_{H_t})=\emptyset$
since $\vvS_{\omega}^{H_{t+1}}(c_{t+1}, r)$ separates $w_{i,t+1}$ from $V_{H_t}$ in $H_{t+1}$. \item For $i,j \leq t$ with $i \neq j$: $E_G(w_{i,t+1}, \gamma_j \cap A_j^{t+1})=\emptyset$:

We have \begin{align*} \dist_{\omega}^{H_{t+1}}(w_{i,t+1}, \gamma_j \cap A_j^{t+1}) & \stackrel{\eqref{eq:wrep-dist}}{\geq} \dist_{\omega}^{G}(r_i^t, r_j^t) - 9.5 \Delta - \len_{\omega}(\gamma_j \cap A_j^{t+1}) \\ &\geq \dist_{\omega}^G(r_i^t, r_j^t) - 18.5 \Delta \\ &\stackrel{\textrm{(P5)}}{\geq} 21(h-t)\Delta - 18.5 \Delta > \Delta\,, \end{align*} which implies the desired bound using \pref{lem:noedge}.

\item $E_G(W^t, \bigcup_{i \leq t} \gamma_i \cap A_i^{t+1})=\emptyset$:  This follows from \pref{lem:far-prelim}. \end{itemize}

\item For $i \leq t$, we have \[\dist_{\omega}^G(r_i^{t+1}, c_i) \leq \dist_{\omega}^G(r_i^t,c_i) + \dist_{\omega}^G(r_i^t, r_i^{t+1}) \leq 10 t\Delta + 9.5\Delta \leq 10(t+1)\Delta.\] Moreover, $r_{t+1}^{t+1}=c_{t+1}$.

\item Similarly, for $i,j \leq t$ and $i\neq j$, \[ \dist_{\omega}^G(r_i^{t+1}, r_j^{t+1}) \geq \dist_{\omega}^G(r_i^t, r_j^t) - 2\cdot 9.5\Delta \stackrel{\textrm{(P5)}}{\geq} 21(h-t-1)\Delta\,. \] One also has \[ \dist_{\omega}^G(r_{t+1}^{t+1}, r_i^{t+1}) = \dist_{\omega}^G(c_{t+1}, r_i^{t+1}) \stackrel{\textrm{(P4)}}{\geq} \dist_{\omega}^G(c_{t+1}, c_i) - 10t\Delta \stackrel{\eqref{eq:center-sep}}{\geq} 21(h-t-1)\Delta\,. \]

\item First, note that if $i \leq t$, then $r_i^{t+1} \in \vvS_{\omega}^{H_{t+1}}(c_{t+1}, r-4\Delta)$, so using (A1) gives \[\dist_{\omega}^{H_{t+1}}(c_{t+1}, r_i^{t+1}) \leq r-4\Delta+\Delta/2\,.\] On the other hand, $(W^t \cup \bigcup_{j \leq t} A_j^t) \cap \vvB_{\omega}^{H_{t+1}}(c_{t+1}, r) = \emptyset$, hence \[ \dist_{\omega}^{H_{t+1}}\left(c_{t+1}, W^t \cup {\textstyle \bigcup_{j \leq t} A_j^t}\right) > r\,. \]
It follows from the triangle inequality
that $\dist_{\omega}^{H_{t+1}}(r_i^{t+1}, W^t \cup \bigcup_{j \leq t} A_j^t) \geq 3\Delta$.

Next, we have, for $i,j \leq t$ and $i \neq j$, \begin{align*} \dist_{\omega}^{H_{t+1}}(r_i^{t+1}, \gamma_j \cap A_j^{t+1}) &\geq \dist_{\omega}^{G}(r_i^{t+1}, \gamma_j \cap A_j^{t+1}) \\ &\geq \dist_{\omega}^G(r_i^t, r_j^t) - \dist_{\omega}^G(r_i^t, r_i^{t+1}) - \len_{\omega}(\gamma_j \cap A_j^{t+1}) \\ &\geq 21 (h-t)\Delta - 9.5\Delta - 5\Delta \geq 3\Delta\,. \end{align*} Also note that $\dist_{\omega}^{H_{t+1}}(c_{t+1}, W_{t+1} \setminus W_t) \leq r-8\Delta+\Delta/2$ because $W^{t+1} \setminus W_t \subseteq \vvS_{\omega}^{H_{t+1}}(c_{t+1},r-8\Delta)$ and $\dist_{\omega}^{H_{t+1}}(c_{t+1}, r_i^{t+1}) \geq r-4\Delta-\Delta/2$ because $r_i^{t+1} \in \vvS_{\omega}^{H_{t+1}}(c_{t+1},r-4\Delta)$. It follows that \[ \dist_{\omega}^{H_{t+1}}(r_i^{t+1}, W^{t+1}\setminus W^t) \geq 3\Delta\,. \] We have thus verified that for $i\leq t$, it holds that $\dist_{\omega}^{H_{t+1}}(r_i^{t+1}, W^{t+1} \cup \bigcup_{j \leq t, j \neq i} A_j^{t+1}) \geq 3\Delta$.

The fact that $\dist_{\omega}^{H_{t+1}}(r_i^{t+1}, A_{t+1}^{t+1}) \geq 3\Delta$ for $i \leq t$ follows
 similarly since $A_{t+1}^{t+1} = {\cB}_{\omega}^{H_{t+1}}(c_{t+1}, r-8\Delta)$
and $r_i^{t+1} \in \vvS_{\omega}^{H_{t+1}}(c_{t+1},r-4\Delta)$.

We are left to verify the last case:  $\dist_{\omega}^{H_{t+1}}(r_{t+1}^{t+1}, W^{t+1} \cup \bigcup_{i \leq t} A_i^{t+1}) \geq 3\Delta$. This follows from the two facts:  $\left(W^{t+1} \cup \bigcup_{i \leq t} A_i^{t+1}\right) \cap {\cB}_{\omega}^{H_{t+1}}(c_{t+1},r-8\Delta) = \emptyset$ and $r \geq 20\Delta$ from \eqref{eq:rbound}. \end{enumerate} We have completed verification of the inductive step, and thus by induction there exists a careful $K_h$ minor in $G$, completing the proof. \end{proof}

\section{Applications and discussion} \label{sec:applications}

\subsection{Spectral bounds} \label{sec:eigenvalues}

Say that a conformal graph $(G,\omega)$ is {\em $(r,\e)$-spreading} if it holds that for every subset $S \subseteq V_G$ with $|S|=r$, one has \[ \frac{1}{|S|^2} \sum_{u,v \in S} \dist_{\omega}(u,v) \geq \e \|\omega\|_{\ell_2(V_G)}\,. \] Let $\e_r(G,\omega)$ be the smallest value $\e$ for which $(G,\omega)$ is $(r,\e)$-spreading. The next theorem appears as \cite[Thm 2.3]{KLPT09}.

\begin{theorem}\label{thm:klpt-evs} If $G$ is an $n$-vertex graph with maximum degree $d_{\max}$, then for $k=1,2,\ldots,n-1$, the following holds: If $\omega : V_G \to \R_+$ satisfies $\|\omega\|_{\ell_2(V_G)}=1$, and $(V_G, \dist_{\omega})$ admits an $(\alpha, \e/2)$-padded partition with $\e = \e_{\lfloor n/8k\rfloor}(G,\omega)$, then \[ \lambda_k(G) \lesssim d_{\max} \frac{\alpha^2}{\e^2 n}\,. \] \end{theorem}

The methods of \cite{KLPT09} also give a way of producing $(r,\e)$-spreading weights. Consider a graph $G$ and let $\mu$ be a probability measure on subsets of $V_G$. A flow $\Lambda : \cP_G \to \R_+$ is called a {\em $\mu$-flow} if \[ \Lambda[u,v] \geq \mu(\left\{ S : u,v \in S\right\})\,. \] For a number $r \leq |V_G|$, let $\cF_r(G)$ denote the set of all $\mu$-flows in $G$ with $\supp(\mu) \subseteq {V_G \choose r}$ (i.e., $\mu$ is supported on subsets of size exactly $r$). The following is a consequence of the duality theory of convex programs (see \cite[Thm 2.4]{KLPT09}).

\begin{theorem}\label{thm:klpt-duality} For every graph $G$ and $r \leq |V_G|$, it holds that \[ \max_{\omega : V_G \to \R_+} \left\{ \e_r(G,\omega) : \|\omega\|_{\ell_2(V_G)} \leq 1\right\} = \frac{1}{r^2} \min \left\{ \|c_{\Lambda}\|_{\ell_2(V_G)} : \Lambda \in \cF_r(G) \right\}\,. \] \end{theorem}

We need to extend the notion of $H$-flows to weighted graphs. Suppose that $H$ is equipped with a non-negative weight on edges $\mathfrak{D} : E_H \to \R_+$. Then an {\em $(H,\mathfrak{D})$-flow in $G$} is a pair $(\Lambda,\f)$ that satisfies properties (1) and (2) of an $H$-flow, but property (3) is replaced by:  For every $u,v \in V_G$, \[ \Lambda[u,v] = \sum_{\{x,y\} \in E_H : \{\f(x),\f(y)\}=\{u,v\}} \mathfrak{D}(\{x,y\})\,. \] We define the {\em crossing congestion $\cross^*_G(H,\mathfrak{D})$} as the infimum of $\cross_G(\Lambda,\f)$ over all $(H,\mathfrak{D})$-flows $(\Lambda,f)$ in $G$. Given a measure $\mu$ on $2^{V_H}$, let $\mathfrak{D}_{\mu}$ be defined by \[ \mathfrak{D}_{\mu}(\{x,y\}) = \mu\left(\left\{S \subseteq V_H : x,y \in S\right\}\right)\,. \] We need the following result which is an immediate consequence of Corollary 3.6 and Corollary 4.2 in \cite{KLPT09}.

\begin{theorem}\label{thm:klpt-main} There is a constant $\theta_0 > 0$ such that for every $h \geq 3$ and $r \geq \theta_0 h^2 \log h$, the following holds. If $G$ excludes $K_h$ as a minor, then for any graph $H$ and any measure $\mu$ supported on ${V_H \choose r}$, it holds that \[ \cross_G^*(H,\mathfrak{D}_{\mu}) \gtrsim \frac{r^5}{|V_H| h^2 \log h}\,. \] \end{theorem}

We can use the preceding theorem combined with the method of \pref{sec:vcon-rigs} to reach a conclusion for rigs over $K_h$-minor-free graphs.

\begin{corollary}\label{cor:c2} Suppose that $\hat{G} \in \rig(G)$ and $G$ excludes $K_h$ as a minor for some $h \geq 3$. Then for every $r \geq \theta_0 h^2 \log h$ and $\Lambda \in \cF_r(\hat G)$, it holds that \[ \|c_{\Lambda}\|^2_{\ell_2(V_{\hat G})} \gtrsim \frac{r^5}{d_{\max}(\hat G) |V_{\hat G}| h^2 \log h}\,. \] \end{corollary}

\begin{proof} Suppose that $\Lambda \in \cF_r(\hat G)$.  Let $(\check{\Lambda},\check{\f})$ be the flow induced in $G$ from the mapping described in the proof of \pref{thm:mat}.  By \pref{claim:crossup}, it holds that \[ \cross_{G}(\check{\Lambda},\check{\f}) \leq 4 d_{\max}(\hat G) \|c_{\Lambda}\|_{\ell_2(V_{\hat G})}^2\,. \] But from \pref{thm:klpt-main}, we know that \[ \cross_G(\check{\Lambda},\check{\f}) \gtrsim \frac{r^5}{|V_{\hat G}|}\,.\qedhere \] \end{proof}

We are now in position to prove \pref{thm:eigenvalues}.

\begin{theorem}[Restatement of \pref{thm:eigenvalues}] Suppose that $G \in \rig(G_0)$ and $G_0$ excludes $K_h$ as a minor for some $h \geq 3$. If $d_{\max}$ is the maximum degree of $G$, then for any $k=1,2,\ldots,|V_G|-1$, it holds that \[ \lambda_k(G) \leq O(d_{\max}^2 h^6 \log h) \frac{k}{|V_G|}\,. \] \end{theorem}

\begin{proof} Let $r=\lfloor n/8k\rfloor$.  We may assume that $r \geq \theta_0 h^2 \log h$ since the bound $\lambda_k(G) \leq 2 d_{\max}(G)$ always holds. From the conjunction of \pref{cor:c2} and \pref{thm:klpt-duality}, we know there exists a conformal metric $\omega : V_G \to \R_+$ with $\|\omega\|_{\ell_2(V_G)}=1$ and such that \[
\e_r(G,\omega) \gtrsim \frac{1}{r^2} \sqrt{\frac{r^5}{d_{\max} h^2 \log h |V_{\hat G}|}}\,. %
\] From \pref{thm:pad}, we know that $(V_G,\dist_{\omega})$ admits an $(\alpha,\Delta)$-padded partition for every $\Delta > 0$ with $\alpha \leq O(h^2)$.  Now applying \pref{thm:klpt-evs} yields the claimed eigenvalue bound. \end{proof}

\subsection{Weighted separators}

Throughout the paper, we have equipped graphs with the uniform measure over their vertices. There are natural extensions to the setting where a graph $G$ is equipped with a non-negative measure on vertices $\mu : V_G \to \R_+$. The corresponding definitions naturally replace $L^p(V_G)$ by the weighted space $L^p(V_G,\mu)$.  The methods of \pref{sec:separators} and \pref{sec:flows} extend in a straightforward way to this setting (see \cite{FHL08} and, in particular, Section 3.6 there for extensions to a more general setting with pairs of weights).

As an illustration, we state a weighted version of \pref{thm:main}. Suppose that $\mu$ is a probability measure on $V_G$. A $\frac23$-balanced separator in $(G,\mu)$ is a subset of nodes $S \subseteq V_G$ such that every connected component of $G[V \setminus S]$ has $\mu$-measure at most $\frac23$.

\begin{theorem} If $G \in \rig(G_0)$ and $G_0$ excludes $K_h$ as a minor, then for any probability measure $\mu$ on $V_G$, there is a $\frac23$-balanced separator of weight at most $c_h \frac{\sqrt{m}}{n}$, where $m=|E_G|$ and $n=|V_G|$. One has the estimate $c_h \leq O(h^3 \log h)$. \end{theorem}

\subsection{Bi-Lipschitz embedding problems}

We state two interesting open metric embedding problems. We state them here only for string graphs, but the extension to rigs over $K_h$-minor-free graphs is straightforward.

\medskip \noindent {\bf Random embeddings into planar graphs.} Let $G$ be a graph and consider a random variable $(\bm{F},\bm{G}_0, \bm{\len})$, where $\bm{F} : V_G \to V_{\bm{G}_0}$, $\bm{G}_0$ is a (random) planar graph, and $\bm{\len} : E_{\bm{G}_0} \to \R_+$ is an assignment of lengths to the edges of $\bm{G}_0$. We use $\dist_{(\bm{G}_0, \bm{\len})}$ denote the induced shortest-path distance in $\bm{G}_0$.

\begin{question}\label{ques:random-embeddings} Is there a constant $K > 0$ so that the following holds for every finite string graph $G$? For every $\omega : V_G \to \R_+$, there exists a triple $(\bm{F},\bm{G}_0, \bm{\len})$ such that: \begin{enumerate} \item (Non-contracting) Almost surely, for every $u,v \in V_G$, \[\dist_{(\bm{G}_0, \bm{\len})}(\bm{F}(u),\bm{F}(v)) \geq
    \dist_{\omega}(u,v)\,.\]
\item (Lipschitz in expectation) For every $u,v \in V_G$, \[ \E\left[\dist_{(\bm{G}_0, \bm{\len})}(\bm{F}(u),\bm{F}(v))\right] \leq K \cdot \dist_{\omega}(u,v)\,. \] \end{enumerate} \end{question}

A positive answer would clarify the geometry of the conformal metrics on string graphs. In \cite{CJLV08}, the lower bound method of \cite{GNRS99} is generalized to rule out the existence of non-trivial reductions in the topology of graphs under random embeddings of the above form. But that method relies on the initial family of graphs being closed under $2$-sums, a property which is manifestly violated for string graphs (since, in particular, string graphs are not closed under subdivision).

\medskip \noindent {\bf Bi-Lipschitz embeddings into $L^1$.} A well-known open question is whether every planar graph metric admits an embedding into $L^1([0,1])$ with bi-Lipschitz distortion at most $C$ (for some universal constant $C$); see \cite{GNRS99} for a discussion of the conjecture and its extension to general excluded-minor families. The following generalization is also natural.

\begin{question}\label{ques:L1strings} Do conformal string metrics admit bi-Lipschitz embeddings into $L^1$? More precisely, is there a constant $K > 0$ such that the following holds for every string graph $G$? For every $\omega : V_G \to \R_+$, there is a mapping $\f : V_G \to L^1([0,1])$ such that for all $u,v \in V_G$, \[ \dist_{\omega}(u,v) \leq \|\f(u)-\f(v)\|_{L^1} \leq K \cdot \dist_{\omega}(u,v)\,. \] \end{question}

Note that, unlike in the case of edge-capacitated flows, a positive resolution does not imply an $O(1)$ vertex-capacitated multi-flow min/cut theorem for string graphs.  See \cite{FHL08} for a discussion and \cite{LMM15} for stronger types of embeddings that do yield this implication.
If \pref{ques:random-embeddings} has a positive resolution, it implies that \pref{ques:L1strings} is equivalent to the same question for planar graphs.

\subsection*{Acknowledgements}

The author thanks Noga Alon, Nati Linial, and Laci Lov\'asz for helpful discussions, Janos Pach for emphasizing Jirka's near-optimal bound for separators in string graphs, and the organizers of the ``Mathematics of Ji{\v{r}}\'i Matou{\v{s}}ek'' conference, where this work was initiated.

\bibliographystyle{alpha} \bibliography{strings}

\end{document}